\documentclass[11pt]{amsart}
\usepackage{amsmath, amsthm, amscd, amsfonts,amssymb, graphicx}
\usepackage{accents}
\usepackage{xcolor}
\usepackage{hyperref}

\newlength{\dhatheight}

\newcommand{\bea}{\begin{eqnarray*}}
\newcommand{\eea}{\end{eqnarray*}}

\newcommand{\beq}{\begin{equation}}
\newcommand{\eeq}{\end{equation}}
\newcommand{\bfomega}{\mbox{\boldmath $\omega$ \unboldmath} \hskip -0.05 true in}

\newcommand{\bom}{\bfomega}

 \newcommand{\ii}{\mathrm{i}}

\newtheorem{theorem}{Theorem}[section]

\theoremstyle{definition}

\newtheorem{proposition}[theorem]{Proposition}
\newtheorem{corollary}[theorem]{Corollary}

\theoremstyle{remark}
\newtheorem{remark}[theorem]{Remark}

\numberwithin{equation}{section}

\begin{document}

\title[Discrete Spectra of Convolutions on Disks using S-L Theory]{Discrete Spectra of Convolutions on Disks using Sturm-Liouville Theory}

\author[A. Ghaani Farashahi]{Arash Ghaani Farashahi$^*$}
\address{Laboratory for Computational Sensing and Robotics (LCSR), Whiting School of Engineering, Johns Hopkins University, Baltimore, Maryland, United States.}
\email{arash.ghaanifarashahi@jhu.edu}
\email{ghaanifarashahi@outlook.com}

\author[G.S. Chirikjian]{Gregory S. Chirikjian}
\address{Laboratory for Computational Sensing and Robotics (LCSR), Whiting School of Engineering, Johns Hopkins University, Baltimore, Maryland, United States.}
\email{gregc@jhu.edu}
\email{gchirik@gmail.com}

\thanks{$^*$Corresponding author}

\begin{abstract}
This paper presents a systematic study for analytic aspects of discrete spectra methods for
convolution of functions supported on disks, according to the Sturm-Liouville theory. We then investigate different aspects of the presented theory in the cases of zero-value boundary condition and derivative boundary condition.
\end{abstract}

\maketitle

\section{{\bf Introduction}}

The theory of convolution operators is placed at the core of many directions of matheatmical analysis such as abstract harmonic analysis, representation theory, functional analysis, and operator theory, see \cite{AGHF.BBMS-NS, AGHF.CJM, AGHF.JAuMS, AGHF.IJM.2015, Kisil.Adv, Kisil.BJMA, 50} and references therein. Over the last decades, some aspects of convolution operators have achieved significant popularity in modern mathematical analysis and constructive approxiamtion areas including Gabor and wavelet analysis \cite{Fei0, Fei1, Fei2, Fei.Gro1, Fei.Gro2} and recent applications in computational science and engineering \cite{PIb2, PIb4, PI6, AGHF.GSG, Kya.PI.2000, Kya.PI.1999}.

Harmonic analysis of functions in $L^2(\mathbb{R}^d)$ and $L^1(\mathbb{R}^d)$ has been extensively developed over the past two centuries.
The Fourier transforms of functions in these spaces have a continuous spectrum. This leads to two well-known problems. First, a function
cannot have compact support in both the real space and in Fourier space. Second, despite the exactness of the theory, when it comes to
 computations, the spectrum must be discretized in some way. These problems are significant because in many applications in engineering, convolutions and
correlations of functions on Euclidean spaces are required. This includes template matching in image processing for pattern recognition, and protein docking \cite{Kim.Kim, 19Venkatraman2009, 19Venkatraman2009a}, and
characterizing how error probabilities propagate \cite{PI5}. Most recently convolutional neural networks (CNNs) in the area of ``deep learning'' use convolution for image recognition problems.

In some applications, the goal is not to recover the values of convolved functions, but rather their support, which is the Minkowski sum of the supports of the two non-negative functions being convolved \cite{Kavraki}.
In many applications the functions of interest
take non-negative values, and as such can be normalized and treated as probability density functions (pdfs).
In this context, convolution of pdfs has the significance that it produces the pdf of the sum of the two random varaiables described
by the original pdfs.

Usually two approaches to circumvent the continuous spectrum are taken to assist in 
computing convolutions of pdfs on Euclidean space. First, if the
functions are compactly supported, then their supports are enclosed in a solid cube with dimensions at least twice the size of the support of the functions, and  periodic versions of the functions are constructed.
In this way, convolution of these periodic functions on the $d$-torus can be used to replace convolution on $d$-dimensional Euclidean space. The benefit of this is that the spectrum is discretized and Fast Fourier Transform (FFT) methods can be used to compute the convolutions. This approach is computationally attractive, but in this periodization procedure the natural invariance of integration
on Euclidean space under rotation transformations is lost when moving to the torus. This can be a significant issue in rotation matching problems.

A second approach is to take the original compactly supported functions and replace them with functions on Euclidean space that have rapidly decaying
tails, but for which convolutions can be computed in closed form. For example, replacing each of the given functions with a sum of Gaussian distributions allows the convolution of the given functions to be computed as
a sum of convolution of Gaussians, which have simple closed-form expressions as Gaussians. The problem with this approach is that the resulting functions are not compactly supported. Moreover, if $N$ Gaussians are used to describe each input function, then $N^2$ Gaussians result after the convolution.

An altogether different approach is explored here. Rather than periodizing the given functions, or extending their support to the whole of Euclidean space, we consider functions that are supported on disks in the
plane (and by natural extension, to balls in higher dimensional Euclidean spaces). The basic idea is that
in polar coordinates each function is expanded in an orthonormal basis
consisting of Bessel functions in the radial direction
and Fourier basis in the angular direction. These basis
elements are orthonormal on the unit disk. Each input function to the convolution procedure is scaled to have support
on the disk of radius of one half and zero-padded on the unit disk. The result of the convolution (or correlation) then is a function which is supported on the unit disk. Since the convolution integral for compactly supported functions can be restricted from all of Euclidean space to the support of the functions,
it is only this integral over the support which is performed when using Fourier-Bessel expansions. Hence, the behavior of these functions outside of disks becomes irrelevant to the final result. Moreover, since these expansions are defined in polar coordinates, they behave naturally under rotations, and result in easily characterization of rotation invariants. 
We work out how the Fourier-Bessel coefficients of the original functions appear in the convolution.

This article contains 5 sections. Section
2 is devoted to establishing notation and gives a
brief summary of convolution of functions on $\mathbb{R}^2$, and polar Fourier analysis associated to Sturm-Liouville theory. In Section
3, we study zero-padded convolutions on disks. Next we present analytic aspects of the general theory of Fourier-Bessel series, that is the discrete spectra associated to zero-valued boundary condition in Sturm-Liouville theory,
for functions defined on disks. We then employ this theory for convolutions on disks.
Section 5 is dedicated to study discrete spectra associated to derivative boundary condition in Sturm-Liouville theory for functions supported on disks. We then apply this method convolutions on disks as well.

\section{{\bf Preliminaries and Notations}}

Throughout this section we shall present preliminaries and the notation.
\subsection{General Notations} For $a>0$, let $\mathbb{B}_a^d:=\{\mathbf{x}\in\mathbb{R}^d:\|\mathbf{x}\|_2\le a\}$. We then put $\mathbb{B}^d:=\mathbb{B}_1^d$, that is the unit ball in $\mathbb{R}^d$.

Let $f_j\in L^1(\mathbb{R}^d)$ with $j\in\{1,2\}$ and $\mathrm{supp}(f_j)\subseteq\mathbb{B}_{a/2}^d$. Then, we have
$$\mathrm{supp}(f_1\ast f_2)\subseteq\mathrm{supp}(f_1)+\mathrm{supp}(f_2)\subseteq\mathbb{B}_{a/2}+\mathbb{B}_{a/2}\subseteq\mathbb{B}^d_a,$$ 
where $+$ denotes the Minkowski sum, and 
for $\mathbf{x}\in\mathbb{R}^d$, we have
\begin{equation}\label{conv.Rd}
(f_1\ast f_2)(\mathbf{x}):=\int_{\mathbb{R}^d}f_1(\mathbf{y})f_2(\mathbf{x}-\mathbf{y})d\mathbf{y}.
\end{equation}

It should be mentioned that, each function $f\in L^1(\mathbb{R}^d)$, satisfies the following integral decomposition;
\begin{equation}\label{int.dec}
\int_{\mathbb{R}^d}f(\mathbf{x})d\mathbf{x}=\int_{\mathbb{S}^{d-1}}\int_0^\infty f(r\mathbf{u})r^{d-1}drd\mathbf{u}.
\end{equation}
Also, if $f\in L^1(\mathbb{R}^d)$ is supported in $\mathbb{B}_a^d$, we then have
\[
\int_{\mathbb{R}^d}f(\mathbf{x})d\mathbf{x}=\int_{\mathbb{S}^{d-1}}\int_0^af(r\mathbf{u})r^{d-1}drd\mathbf{u}.
\]

The Fourier transform of of the function $f\in L^1(\mathbb{R}^d)$ is given by
\begin{equation}
\widehat{f}(\bom)=\int_{\mathbb{R}^d}f(\mathbf{x})e^{-\ii\bom\cdot\mathbf{x}}d\mathbf{x},
\end{equation}
for all $\bom\in\mathbb{R}^d=\widehat{\mathbb{R}^d}$.

We then have the following reconstruction formula
\begin{equation}
f(\mathbf{x})=(2\pi)^{-d}\int_{\mathbb{R}^d}\widehat{f}(\bom)e^{\ii\bom\cdot\mathbf{x}}d\bom,
\end{equation}
for $\mathbf{x}\in\mathbb{R}^d$.

\subsection{\bf Polar Fourier Analysis for the Case When $d=2$}
For $s$ taking continuous non-negative values and integer $m\in\mathbb{Z}$, the basis functions (polar harmonics) are given by
\begin{equation}\label{2D.B.Wh}
\Phi_{s}^{m}(r\mathbf{u})=\Phi_{s}^{m}(r,\theta)=\sqrt{s}J_m(sr)\cdot \mathcal{Y}_m(\theta),
\end{equation}
where $\mathcal{Y}_m(\theta):=\frac{1}{\sqrt{2\pi}}\exp(\ii m\theta)$ and $J_m(x)$ is an $m^{th}$ order Bessel function of the first kind. 

Then, any 2D function $f(r,\theta)$ defined on the whole space (i.e. $0\le \theta\le 2\pi$ and $0\le r <\infty$) can be expanded with respect to $\Phi_s^m$ as defined in (\ref{2D.B.Wh}) via
\begin{equation}\label{2D.EX.Wh}
f(r,\theta)=\int_0^\infty\left(\sum_{m=-\infty}^{+\infty}\mathrm{c}_{s,m}(f)\Phi_s^m(r,\theta)\right)sds,
\end{equation}
where
\begin{equation}
\mathrm{c}_{s,m}(f):=\int_0^\infty\int_0^{2\pi}f(r,\theta)\overline{\Phi_s^m(r,\theta)}rdr d\theta.
\end{equation}

The expansion (\ref{2D.EX.Wh}) is mainly of theoretical interest, since there is an integral in the reconstruction formula.
But in experiment, one should use expansions valid/defined on finite regions, in order to make the integral in the reconstruction formula (\ref{2D.EX.Wh}) as a structured discrete sum. In this direction, we need to redefine basis functions.

Let $a>0$ and $\alpha,\beta\in\mathbb{R}$. We then have
\begin{equation}\label{2J.a}
\int_0^aJ_m(\alpha r)J_m(\beta r)rdr=\frac{\beta J_m(\alpha a)J_m'(\beta a)-\alpha J_m(\beta a)J_m'(\alpha a)}{a^{-1}(\alpha^2-\beta^2)}.
\end{equation}

The eigenvalue problem can be written as
\begin{equation}\label{Helm.PL}
\nabla_r^2\Psi(r,\theta)+\frac{1}{r^2}\nabla_\theta^2\Psi(r,\theta)+\rho^2\Psi(r,\theta)=0,
\end{equation}
which is the Helmholtz differential equation in polar coordinates, with $\rho\ge 0$, where
\begin{equation}
\nabla_r^2:=\frac{1}{r}\frac{\partial}{\partial r}\left(r\frac{\partial r}{r}\right),\hspace{1cm}\nabla_\theta^2:=\frac{\partial^2}{\partial\theta^2}.
\end{equation}
Substituting the separation of variable form $\Psi(r,\theta)=v(r)u(\theta)$ into (\ref{Helm.PL}), we get
\begin{equation}\label{V.part}
\nabla_\theta^2u+m^2u=0,
\end{equation}
and
\begin{equation}\label{U.part}
\nabla_r^2v+\left(\rho^2-\frac{m^2}{r^2}\right)v=0.
\end{equation}

Using appropriate boundary conditions according to the Sturm-Liouville theory, a set of $\rho$ values can be determined that makes $J_m(\rho r)$ again mutually orthogonal. In this direction,
(\ref{U.part}), can be rewritten as
\begin{equation}\label{V.part.SL0}
-(rv')'+\frac{m^2}{r}v=\rho^2 rv.
\end{equation}
Therefore, with
\[
P(r):=r,\hspace{1cm}Q(r):=-\frac{m^2}{r},\hspace{1cm}w(r):=r,\hspace{1cm}\lambda:=\rho^2,
\]
the equation (\ref{V.part.SL0}) takes the following S-L form
\begin{equation}\label{V.part.SL}
(P(r)v')'+Q(r)v=-\lambda w(r)v,
\end{equation}
with $r\in [0,a]$. Then, Equation (\ref{V.part.SL}), with respect to the following boundary conditions forms a singular S-L system;
\begin{equation}\label{SL.BVC0}
v(0)\cos\psi-p(0)v'(0)\sin\psi=0,
\end{equation}
and
\begin{equation}\label{SL.BVC1}
v(a)\cos\phi-p(a)v'(a)\sin\phi=0,
\end{equation}
with $\psi,\phi\in [0,\pi)$.

Invoking Sturm-Liouville theory, for such S-L problems \cite{SLT.IA, Zet.An}, we have
\begin{enumerate}
\item The eigenvalues $\lambda_1,\lambda_2,\lambda_3,\cdots$ of this S-L problem
are nonnegative real and can be ordered such that
$$\lambda_1<\lambda_2<\lambda_3<\cdots<\lambda_n<\cdots$$
with $\lim_{n\to\infty}\lambda_n=\infty$.
\item Corresponding to each eigenvalue $\lambda_n$ is a unique (up to a normalization constant) eigenfunction $y_n$ which has exactly $n-1$ zeros in $(0,a)$. The eigenfunction $y_n$ is called the $n$-th fundamental solution satisfying the Sturm-Liouville problem.
\item The normalized eigenfunctions form an orthonormal basis with respect to the weight function $w(r):=r$, that is
\begin{equation}
\int_0^ay_n(r)y_{n'}(r)rdr=\delta_{nn'},
\end{equation}
in the Hilbert function space $L^2([0,a],w(r)dr)$, where $\delta_{nn'}$ is the Kronecker delta.
\end{enumerate}

With $\psi:=\pi/2$, the equation (\ref{SL.BVC0}), has no effect on the selection of $\rho$. Thus, the only effective boundary condition is (\ref{SL.BVC1}). Substituting $v(r):=J_m(\rho r)$ into (\ref{SL.BVC}), we get
\[
J_m(\rho r)\cos\phi-\rho aJ_m'(\rho r)\sin\phi=0,
\]
with $x:=\rho r$, we have
\[
J_m(x)\cos\phi-xJ_m'(x)\sin\phi=0.
\]

Let $a>0$, and $m\in\mathbb{Z}$.
Suppose $\{z_{mn}:n\in\mathbb{N}\}$ is the set of all nonnegative zeros of
\begin{equation}\label{SL.BVC}
J_m(x)\cos\phi-xJ_m'(x)\sin\phi=0,
\end{equation}
such that $z_{m1}<z_{m2}<\cdots<z_{mn}<\cdots$, i.e.
\[
J_m(z_{mn})\cos\phi-z_{mn}J_m'(z_{mn})\sin\phi=0,
\]
for all $n\in\mathbb{N}$. For $n\in\mathbb{N}$, let $\rho_{nm}$ be given by
$\rho_{nm}:=a^{-1}z_{mn}$. We then have
\[
J_m(\rho_{nm}a)\cos\phi-a\rho_{nm}J_m'(a\rho_{nm})\sin\phi=0,
\]
for all $n\in\mathbb{N}$.

In this case, the $n$-th eigenvalue is then $\lambda_n:=\rho_{nm}^2$ and the $n$-th
eigenfunction is $J_m(\rho_{nm}r)$. Invoking Sturm-Liouville theory, the orthogonality of the eigenfunctions can be written as
\begin{equation}
\int_0^aJ_{m}(\rho_{nm}r)J_{m}(\rho_{n'm}r)rdr=N_{n}^{(m)}(a)\delta_{nn'},
\end{equation}
for all $n,n'\in\mathbb{N}$, where
\begin{equation}
N_n^{(m)}(a)=\frac{a^2}{2}\left(J_m'(z_{mn})^2+\left(1-\frac{m^2}{z_{mn}^2}\right)J_m(z_{mn})^2\right).
\end{equation}
We then have
\[
N_n^{(m)}(a)=a^2D_n^{(m)},
\]
with
\[
D_n^{(m)}:=N_n^{(m)}(1)=\frac{1}{2}\left(J_m'(z_{mn})^2+\left(1-\frac{m^2}{z_{mn}^2}\right)J_m(z_{mn})^2\right).
\]
The generalized (redefined) basis functions (polar harmonics) are given by
\begin{equation}\label{2D.B.Fr.g}
\Psi_{nm}^a(r\mathbf{u}_\theta)=\Psi_{nm}^a(r,\theta):=\mathcal{J}_{nm}^a(r)\mathcal{Y}_m(\theta),
\end{equation}
where $\mathcal{Y}_m(\theta):=\frac{1}{\sqrt{2\pi}}\exp(\ii m\theta)$ and the generalized Bessel (normalized radial) function $\mathcal{J}_{nm}^a(r)$ is given by
\begin{equation}
\mathcal{J}_{nm}^a(r):=\frac{1}{\sqrt{N_n^{(m)}(a)}}J_m(\rho_{nm}r)=\frac{1}{\sqrt{N_n^{(m)}(a)}}J_m(a^{-1}z_{mn}r),
\end{equation}
Hence, we get
\begin{equation}
\int_0^a\mathcal{J}_{nm}^a(r)\mathcal{J}_{n'm}^a(r)rdr=\delta_{nn'}.
\end{equation}
Therefore, for each $m\in\mathbb{Z}$, the set $\mathcal{B}_m^a:=\{\mathcal{J}_{nm}^a:n\in\mathbb{N}\}$ forms an orthonormal basis on the interval $[0,a]$. 
Hence, any function $v:[0,a]\to\mathbb{R}$ with $v\in L^2([0,a],rdr)$, which equivalently means
\[
\int_0^a|v(r)|^2rdr<\infty,
\]
satisfies the following constructive expansion
\begin{equation}\label{mFBS}
v(r)=\sum_{n=1}^\infty\left(\int_0^av(s)\mathcal{J}_{nm}^a(s)sds\right)\mathcal{J}_{nm}^a(r).
\end{equation}
We then conclude that
\begin{equation}
\int_0^a\int_0^{2\pi}\Psi_{nm}^a(r,\theta)\overline{\Psi_{n'm'}^a(r,\theta)}rdrd\theta=\delta_{nn'}\delta_{mm'}.
\end{equation}
Consequently, any restricted 2D square-integrable function $\xi(r,\theta)$ defined on 
the disk $\mathbb{B}_a^2$, that is $\xi\in L^2(\mathbb{B}_a^2)$, can be expanded with respect to $\Psi_{nm}^a$ as defined in (\ref{2D.B.Fr.g}) via
\begin{equation}\label{2D.EX.Fr}
\xi(r,\theta)=\sum_{m=-\infty}^\infty\sum_{n=1}^{\infty} C_{n,m}^a(\xi)\Psi_{nm}^a(r,\theta),
\end{equation}
where
\begin{equation}
C_{n,m}^a(\xi):=\int_0^a\int_0^{2\pi}\xi(r,\theta)\overline{\Psi_{nm}^a(r,\theta)}rdr d\theta.
\end{equation}

\subsubsection{\bf Zero-Value Boundary Condition}\label{ZBC}
If $\sin\phi=0$, then (\ref{SL.BVC}) reduces to
\[
J_m(x)=0.
\]
Thus, $z_{mn}$ should be the positive zeros of $J_m(x)$. We then have
\[
N_n^{(m)}(a)=\frac{a^2}{2}J_{m}'(z_{mn})^2=\frac{a^2}{2}J_{m+1}^2(z_{mn})=a^2D_n^{(m)},
\]
with
\[
D_n^{(m)}:=\frac{J_{m}'(z_{mn})^2}{2}=\frac{J_{m+1}^2(z_{mn})}{2}.
\]
In this case, the right hand side of (\ref{mFBS}) is called as $m$-th order Fourier-Bessel series of $v$.

\subsubsection{\bf Derivative Boundary Condition}\label{DBC}
If $\cos\phi=0$, then (\ref{SL.BVC}) reduces to
\begin{equation}
J_m'(x)=0.
\end{equation}
Thus, $z_{mn}$ should be the zeros of $J_m'(x)=0$. In this case, $x=0$ is one solution to $J_0'(x)=0$. Invoking Sturm-Liouville theory, $x=0$ should be considered as $z_{01}$.
We then have
\begin{equation}\label{Nnm.DBC}
N_n^{(m)}(a)=a^2D_{n}^{(m)}=\frac{a^2}{2}\left(1-\frac{m^2}{z_{mn}^2}\right)J_m^2(z_{mn}),
\end{equation}
with
\begin{equation}\label{DBC0}
N_1^{(0)}(a)=\frac{a^2}{2}.
\end{equation}
In this case, we have
\begin{equation}
D_n^{(m)}=N_n^{(m)}(1)=\frac{1}{2}\left(1-\frac{m^2}{z_{mn}^2}\right)J_m(z_{mn})^2.
\end{equation}
\section{\bf Zero-Padded Convolutions on Disks}

In this section, we study basic analytic aspects of zero-padded convolutions on disks.   

Let $a>0$ and $C_a:=\mathbb{B}_a^2$ be the disk of radius $a$ in $\mathbb{R}^2$.
Let $\xi:C_a\to\mathbb{C}$ be a function.
Then, there exist a canonical extension of $\xi$ from $C_a$ to $\mathbb{R}^2$, mostly denoted by $E(\xi):\mathbb{R}^2\to\mathbb{C}$ such that $E(\xi)(\mathbf{x})=\xi(\mathbf{x})$ for all $\mathbf{x}\in C_a$, and $E(\xi)(\mathbf{x})=0$ for all $\mathbf{x}\not\in C_a$.  If $\xi\in L^p(C_a)$ we then have $E(\xi)\in L^p(\mathbb{R}^2)$, for all $p\ge1$. In particular, if $\xi$ is continuous we then have $E(\xi)\in L^p(\mathbb{R}^2)$, for all $p\ge 1$.

Let $a>0$ and $b:=a/2$. Let $\xi_j:C_b\to\mathbb{C}$ with $j\in\{1,2\}$ be $L^1$-functions on $C_b$. We then define the canonical $C_b$-windowed convolution of $\xi_1$ with $\xi_2$, 
denoted by $\xi_1\circledast\xi_2:\mathbb{R}^2\to\mathbb{C}$, by
\[
\xi_1\circledast\xi_2:=E(\xi_1)\ast E(\xi_2),
\]
where $E(\xi_j)$ is the canonical extension of $\xi_j$ from $C_a$ to $\mathbb{R}^2$ by zero-padding. That is 
\begin{equation}
(\xi_1\circledast\xi_2)(\mathbf{x}):=(E(\xi_1)\ast E(\xi_2))(\mathbf{x})=\int_{\mathbb{R}^2}E(\xi_1)(\mathbf{y})E(\xi_2)(\mathbf{x}-\mathbf{y})d\mathbf{y},
\end{equation}
for all $\mathbf{x}\in \mathbb{R}^2$.

Since each $E(\xi_j)$ is supported in $C_b$, we deduce that $E(\xi_1)\ast E(\xi_2)$ is supported in $C_b+C_b=C_a$. Hence, we get
\[
(\xi_1\circledast\xi_2)(\mathbf{x})=\int_{C_b}E(\xi_1)(\mathbf{y})E(\xi_2)(\mathbf{x}-\mathbf{y})d\mathbf{y}=\int_{C_b}\xi_1(\mathbf{y})E(\xi_2)(\mathbf{x}-\mathbf{y})d\mathbf{y},
\]
for all $\mathbf{x}\in C_a$.

In polar form, we get  
\[
(\xi_1\circledast\xi_2)(\mathbf{x})=\int_{\mathbb{S}^1}\int_0^b\xi_1(r\mathbf{u})E(\xi_2)(\mathbf{x}-r\mathbf{u})rdrd\mathbf{u},
\]
for all $\mathbf{x}\in C_a$.

Let $f_j:\mathbb{R}^2\to\mathbb{C}$ with $j\in\{1,2\}$ be $L^1$-functions on $C_b$, which means that $\xi_j:=R(f_j)$ that is restriction of each $f_j$ to $C_b$, belongs to the $L^1$-function space on $C_b$. We then define the canonical $C_b$-windowed convolution of $f_1$ with $f_2$, denoted by $f_1\circledast f_2:\mathbb{R}^2\to\mathbb{C}$, by  
\begin{equation}
f_1\circledast f_2:=\xi_1\circledast \xi_2=E(\xi_1)\ast E(\xi_2),
\end{equation}
where $\xi_j:=R(f_j)$ is the restriction of $f_j$ to the subset $C_b$ and $E(\xi_j)$ is the canonical extension of $\xi_j$ to $\mathbb{R}^2$ by zero-padding.
That is 
\begin{equation}
(f_1\circledast f_2)(\mathbf{x})=\int_{C_b}f_1(\mathbf{y})E(\xi_2)(\mathbf{x}-\mathbf{y})d\mathbf{y},
\end{equation}
for all $\mathbf{x}\in C_a$.

In polar form, we get  
\[
(f_1\circledast f_2)(\mathbf{x})=\int_{\mathbb{S}^1}\int_0^b f_1(r\mathbf{u})E(\xi_2)(\mathbf{x}-r\mathbf{u})rdrd\mathbf{u},
\]
for all $\mathbf{x}\in C_a$.

The following result presents basic properties of zero-padded functions on disks.
 
\begin{theorem}\label{ZPC.12}
Let $a>0$ and $b:=a/2$. 
\begin{enumerate}
\item Suppose $f_1\in L^1(\mathbb{R}^2)$ and $f_2\in L^1(\mathbb{R}^2)$ be functions. We then have $f_1\circledast f_2\in L^1(\mathbb{R}^2)$ with
\[
\|f_1\circledast f_2\|_{L^1(\mathbb{R}^1)}\le \|f_1\|_{L^1(\mathbb{R}^2)}\|f_2\|_{L^1(\mathbb{R}^2)}.
\]
\item Suppose $f_1\in L^1(\mathbb{R}^2)$ and $f_2\in L^2(\mathbb{R}^2)$ be functions. We then have $f_1\circledast f_2\in L^1(\mathbb{R}^2)$ with
\[
\|f_1\circledast f_2\|_{L^1(\mathbb{R}^1)}\le b\sqrt{\pi}\|f_1\|_{L^1(\mathbb{R}^2)}\|f_2\|_{L^2(\mathbb{R}^2)}.
\]
\item Suppose $f_1\in L^2(\mathbb{R}^2)$ and $f_2\in L^2(\mathbb{R}^2)$ be functions. We then have $f_1\circledast f_2\in L^1(\mathbb{R}^2)$ with
\[
\|f_1\circledast f_2\|_{L^1(\mathbb{R}^1)}\le \pi b^2\|f_1\|_{L^2(\mathbb{R}^2)}\|f_2\|_{L^2(\mathbb{R}^2)}.
\]
\end{enumerate}
\end{theorem}
\begin{proof}
Let $a>0$ and $b:=a/2$. Let $p,q\in\{1,2\}$. Suppose $f_1\in L^p(\mathbb{R}^2)$ and $f_2\in L^q(\mathbb{R}^2)$ be functions. Let $\xi_j:=R(f_j)$ be the restriction of $f_j$ to the disk $\mathbb{B}_b^2$ and $E(\xi_j)$ be the extension of $\xi_j$ to $\mathbb{R}^2$ by zero-padding. Invoking the assumptions $f_1\in L^p(\mathbb{R}^2)$ and $f_2\in L^q(\mathbb{B}^2_b)$, we get $\xi_1\in L^p(\mathbb{B}_b^2)$ and $\xi_2\in L^q(\mathbb{R}^2)$ with 
$$\|\xi_1\|_{L^p(\mathbb{B}_b^2)}\le\|f_1\|_{L^p(\mathbb{R}^2)},$$ 
and 
$$\|\xi_2\|_{L^q(\mathbb{B}_b^2)}\le\|f_2\|_{L^q(\mathbb{R}^2)}.$$
(1) Let $p=q=1$. We then have $E(\xi_j)\in L^1(\mathbb{R}^2)$ with 
\[
\|E(\xi_j)\|_{L^1(\mathbb{R}^2)}=\|\xi_j\|_{L^1(\mathbb{B}_b^2)}.
\]
Hence, we get  
\begin{align*}
\|f_1\circledast f_2\|_{L^1(\mathbb{R}^1)}
&=\|E(\xi_1)\ast E(\xi_2)\|_{L^1(\mathbb{R}^1)}
\\&\le\|E(\xi_1)\|_{L^1(\mathbb{R}^1)}\|E(\xi_2)\|_{L^1(\mathbb{R}^1)}
\\&=\|\xi_1\|_{L^1(\mathbb{B}_b^2)}\|\xi_2\|_{L^1(\mathbb{B}_b^2)}
\le \|f_1\|_{L^1(\mathbb{R}^2)}\|f_2\|_{L^1(\mathbb{R}^2)}.
\end{align*}
(2) Let $p=1$ and $q=2$. Since $\mathbb{B}_b^2$ is compact and hence of finite Lebesgue measure, we have $L^2(\mathbb{B}_b^2)\subseteq L^1(\mathbb{B}_b^2)$. 
Thus, we get $\xi_2\in L^1(\mathbb{B}_b^2)$, with 
\begin{equation}\label{p1q2}
\|\xi_2\|_{L^1(\mathbb{B}_b^2)}\le b\sqrt{\pi}\|\xi_2\|_{L^2(\mathbb{B}_b^2)}\le 
b\sqrt{\pi}\|f_2\|_{L^2(\mathbb{R}^2)},
\end{equation}
Therefore, we have $E(\xi_j)\in L^1(\mathbb{R}^2)$ with 
\[
\|E(\xi_j)\|_{L^1(\mathbb{R}^2)}=\|\xi_j\|_{L^1(\mathbb{B}_b^2)}.
\]
Hence, using Equation (\ref{p1q2}), we can write 
\begin{align*}
\|f_1\circledast f_2\|_{L^1(\mathbb{R}^1)}
&=\|E(\xi_1)\ast E(\xi_2)\|_{L^1(\mathbb{R}^1)}
\\&\le\|E(\xi_1)\|_{L^1(\mathbb{R}^1)}\|E(\xi_2)\|_{L^1(\mathbb{R}^1)}
\\&=\|\xi_1\|_{L^1(\mathbb{B}_b^2)}\|\xi_2\|_{L^1(\mathbb{B}_b^2)}
\le b\sqrt{\pi}\|f_1\|_{L^1(\mathbb{R}^2)}\|f_2\|_{L^2(\mathbb{R}^2)}.
\end{align*}
(3) Let $p=q=2$. Since $L^2(\mathbb{B}_b^2)\subseteq L^1(\mathbb{B}_b^2)$, 
we get $\xi_j\in L^1(\mathbb{B}_b^2)$, with 
\begin{equation}\label{111}
\|\xi_1\|_{L^1(\mathbb{B}_b^2)}\le b\sqrt{\pi}\|\xi_1\|_{L^2(\mathbb{B}_b^2)}\le b\sqrt{\pi}\|f_1\|_{L^2(\mathbb{R}^2)},
\end{equation}
and
\begin{equation}\label{221}
\|\xi_2\|_{L^1(\mathbb{B}_b^2)}\le b\sqrt{\pi}\|\xi_2\|_{L^2(\mathbb{B}_b^2)}\le b\sqrt{\pi}\|f_2\|_{L^2(\mathbb{R}^2)}.
\end{equation}
Therefore, we have $E(\xi_j)\in L^1(\mathbb{R}^2)$ with 
\[
\|E(\xi_j)\|_{L^1(\mathbb{R}^2)}=\|\xi_j\|_{L^1(\mathbb{B}_b^2)}.
\]
Hence, using Equations (\ref{111}) and (\ref{221}), we can write 
\begin{align*}
\|f_1\circledast f_2\|_{L^1(\mathbb{R}^1)}
&=\|E(\xi_1)\ast E(\xi_2)\|_{L^1(\mathbb{R}^1)}
\\&\le\|E(\xi_1)\|_{L^1(\mathbb{R}^1)}\|E(\xi_2)\|_{L^1(\mathbb{R}^1)}
\\&=\|\xi_1\|_{L^1(\mathbb{B}_b^2)}\|\xi_2\|_{L^1(\mathbb{B}_b^2)}
\le b^2\pi\|f_1\|_{L^2(\mathbb{R}^2)}\|f_2\|_{L^2(\mathbb{R}^2)}.
\end{align*}
\end{proof}

\begin{corollary}
{\it Let $p,q\in\{1,2\}$. Suppose $f_j:\mathbb{R}^2\to\mathbb{C}$ with $j\in\{1,2\}$ 
be continuous functions supported in $\mathbb{B}_b^2$. 
We then have $f_1\circledast f_2=f_1\ast f_2$ 
and hence $f_1\circledast f_2\in L^1(\mathbb{R}^2)$ with 
\[
\|f_1\circledast f_2\|_{L^1(\mathbb{R}^1)}\le \|f_1\|_{L^p(\mathbb{R}^2)}\|f_2\|_{L^q(\mathbb{R}^2)}.
\]
In particular, $f_1\circledast f_2$ is continuous and supported in $\mathbb{B}_a^2$.
}\end{corollary}

\begin{theorem}\label{222.disk}
Let $a>0$ and $b:=a/2$. Suppose $\xi_j\in L^2(\mathbb{B}_b^2)$ with $j\in\{1,2\}$ be functions. Then, $\xi_1\circledast\xi_2$ is square integrable on $\mathbb{B}_a^2$ with 
\[
\|\xi_1\circledast \xi_2\|_{L^2(\mathbb{R}^2)}\le b\sqrt{\pi}\|\xi_1\|_{L^2(\mathbb{B}_b^2)}\|\xi_2\|_{L^2(\mathbb{B}_b^2)}.
\]
In particular, if $R(\xi_1\circledast\xi_2)$ is the restriction of $\xi_1\circledast\xi_2$ into $\mathbb{B}_a^2$, we then have 
\[
\|R(\xi_1\circledast \xi_2)\|_{L^2(\mathbb{B}_a^2)}\le b\sqrt{\pi}\|\xi_1\|_{L^2(\mathbb{B}_b^2)}\|\xi_2\|_{L^2(\mathbb{B}_b^2)}.
\]
\end{theorem}
\begin{proof}
Let $a>0$ and $b:=a/2$. Suppose $\xi_j\in L^2(\mathbb{B}_b^2)$ with $j\in\{1,2\}$ be functions. Since $C_b$ is compact, we have 
\begin{align*}
|\xi_1\circledast\xi_2(\mathbf{x})|^2
&=\left|\int_{C_b}\xi_1(\mathbf{y})E(\xi_2)(\mathbf{x}-\mathbf{y})d\mathbf{y}\right|^2
\\&\le\left(\int_{C_b}|\xi_1(\mathbf{y})||E(\xi_2)(\mathbf{x}-\mathbf{y})|d\mathbf{y}\right)^2
\le\pi b^2\left(\int_{C_b}|\xi_1(\mathbf{y})|^2|E(\xi_2)(\mathbf{x}-\mathbf{y})|^2d\mathbf{y}\right),
\end{align*}
for all $\mathbf{x}\in\mathbb{R}^2$. We then have 
\begin{align*}
\|\xi_1\circledast \xi_2\|_{L^2(\mathbb{R}^2)}^2
&=\int_{\mathbb{R}^2}|\xi_1\circledast\xi_2(\mathbf{x})|^2d\mathbf{x}
\\&\le\pi b^2\int_{\mathbb{R}^2}\left(\int_{C_b}|\xi_1(\mathbf{y})|^2|E(\xi_2)(\mathbf{x}-\mathbf{y})|^2d\mathbf{y}\right)d\mathbf{x}
\\&=\pi b^2\int_{C_b}\int_{\mathbb{R}^2}|\xi_1(\mathbf{y})|^2|E(\xi_2)(\mathbf{x}-\mathbf{y})|^2d\mathbf{x}d\mathbf{y}
\\&=\pi b^2\left(\int_{C_b}|\xi_1(\mathbf{y})|^2\left(\int_{\mathbb{R}^2}|E(\xi_2)(\mathbf{x}-\mathbf{y})|^2d\mathbf{x}\right)d\mathbf{y}\right)
\\&=\pi b^2\left(\int_{C_b}|\xi_1(\mathbf{y})|^2d\mathbf{y}\right)\left(\int_{\mathbb{R}^2}|E(\xi_2)(\mathbf{x})|^2d\mathbf{x}\right)
\\&=\pi b^2\left(\int_{C_b}|\xi_1(\mathbf{y})|^2d\mathbf{y}\right)\left(\int_{C_b}|\xi_2(\mathbf{x})|^2d\mathbf{x}\right)
=\pi b^2\|\xi_1\|_{L^2(\mathbb{B}_b^2)}^2\|\xi_2\|_{L^2(\mathbb{B}_b^2)}^2.
\end{align*}
Let $R(\xi_1\circledast\xi_2)$ be the restriction of $\xi_1\circledast\xi_2$ into 
$\mathbb{B}_a^2$. Since $\xi\circledast\xi_2$ is supported in $\mathbb{B}_a^2$, we have 
\[
\|R(\xi_1\circledast \xi_2)\|_{L^2(\mathbb{B}_a^2)}^2=\|\xi_1\circledast \xi_2\|_{L^2(\mathbb{R}^2)}^2\le \pi b^2\|\xi_1\|_{L^2(\mathbb{B}_b^2)}^2\|\xi_2\|_{L^2(\mathbb{B}_b^2)}^2.
\]
\end{proof}

\begin{corollary}
{\it Let $a>0$ and $b:=a/2$. Suppose $\xi_j:\mathbb{B}_b^2\to\mathbb{C}$ 
with $j\in\{1,2\}$ be continuous functions. Then, $\xi_1\circledast\xi_2$ is square integrable on $\mathbb{B}_a^2$ with 
\[
\|\xi_1\circledast \xi_2\|_{L^2(\mathbb{R}^2)}\le b\sqrt{\pi}\|\xi_1\|_{L^2(\mathbb{B}_b^2)}\|\xi_2\|_{L^2(\mathbb{B}_b^2)}.
\]
}\end{corollary}

\begin{proposition}\label{222.R2}
{\it Let $a>0$ and $b:=a/2$. Let $f_j:\mathbb{R}^2\to\mathbb{C}$ with $j\in\{1,2\}$ be functions square integrable on $\mathbb{B}_b^2$.
Then $f_1\circledast f_2$ is square integrable on $\mathbb{B}_a^2$ and we have 
\[
\|R(f_1\circledast f_2)\|_{L^2(\mathbb{B}_a^2)}\le b\sqrt{\pi}\|R(f_1)\|_{L^2(\mathbb{B}_b^2)}\|R(f_2)\|_{L^2(\mathbb{B}_b^2)}.
\]
In particular, if $f_j\in L^2(\mathbb{R}^2)$ with $j\in\{1,2\}$, we have 
\[
\|f_1\circledast f_2\|_{L^2(\mathbb{R}^2)}\le b\sqrt{\pi}\|f_1\|_{L^2(\mathbb{R}^2)}
\|f_2\|_{L^2(\mathbb{R}^2)}.
\]
}\end{proposition}

\begin{corollary}
{\it Let $a>0$ and $b:=a/2$. Let $f_j:\mathbb{R}^2\to\mathbb{C}$ with $j\in\{1,2\}$ 
be continuous functions.
Then $f_1\circledast f_2$ is square integrable on $\mathbb{B}_a^2$ and we have 
\[
\|R(f_1\circledast f_2)\|_{L^2(\mathbb{B}_a^2)}\le b\sqrt{\pi}\|R(f_1)\|_{L^2(\mathbb{B}_b^2)}\|R(f_2)\|_{L^2(\mathbb{B}_b^2)}.
\]
In particular, if $f_j$ with $j\in\{1,2\}$ are compactly supported in $\mathbb{B}_b^2$, we have 
\[
\|f_1\circledast f_2\|_{L^2(\mathbb{R}^2)}\le b\sqrt{\pi}\|f_1\|_{L^2(\mathbb{R}^2)}
\|f_2\|_{L^2(\mathbb{R}^2)}.
\]
}
\end{corollary}

\section{\bf Discrete Spectra on Disks using Zero-Value Boundary Condition}

Throughout this section, for each $m\in\mathbb{Z}$ and $n\in\mathbb{N}$,
we assume that $\rho_{nm}$ are selected with respect to the zero-value boundary condition, according to the Sturm-Liouville theory, see Subsection \ref{ZBC}.

Next we shall present a unified method for computing the coefficients of
convolution functions, if the basis functions are given by (\ref{2D.B.Fr.g}) with respect to zero-valued boundary condition.

First, we need some preliminaries results.

\begin{proposition}
{\it Let $a>0$, $0<r,s\le a$, and $0<\alpha,\theta\le 2\pi$. We then have
\begin{equation}\label{JA.ge}
e^{\ii rs\cos(\alpha-\theta)}=2\sqrt{\pi}\sum_{n=1}^{\infty}\sum_{m=-\infty}^{\infty}
(-1)^n\ii^m\rho_{nm}\frac{J_m(ar)}{r^2-z_{nm}^2}e^{-\ii m\alpha}\Psi_{nm}^a(s,\theta).
\end{equation}
}\end{proposition}
\begin{proof}
Let $a>0$ and $0<r,s\le a$. Also, let $\mathbf{x}:=s\mathbf{u}_\theta$ and $\bom:=r\mathbf{u}_\alpha$. By Jacobi-Anger expansion, we can write
\begin{equation}\label{JA.e0}
e^{\ii rs\cos(\alpha-\theta)}=e^{\ii{\bom}\cdot\mathbf{x}}=\sum_{m=-\infty}^{\infty}\ii^m J_m(rs)e^{\ii m\theta}e^{-\ii m\alpha}.
\end{equation}
Let $m\in\mathbb{Z}$ and $0<r\le a$. Expanding $J_m(rs)$ with respect to $s$ as a function over $[0,a]$, using (\ref{mFBS}), we have
\begin{align*}
J_m(rs)=\sum_{n=1}^\infty\left(\int_0^a\mathcal{J}_{nm}^a(p)J_{m}(rp)pdp\right)\mathcal{J}_{nm}^a(s).
\end{align*}
Using (\ref{2J.a}), and since $z_{mn}$ are selected with respect to zero-valued boundary condition, for each $n\in\mathbb{N}$, we get
\begin{align*}
\int_0^a\mathcal{J}_{nm}^a(p)J_{m}(rp)pdp
&=\frac{1}{\sqrt{N_n^{(m)}(a)}}\int_0^aJ_m(\rho_{nm}p)J_m(rp)pdp
\\&=\frac{a}{\sqrt{N_n^{(m)}(a)}}\frac{rJ_m(\rho_{nm}a)J_m'(ra)-\rho_{nm}J_m(ra )J_m'(\rho_{nm}a)}{(\rho_{nm}^2-r^2)}
\\&=\frac{rJ_m(z_{mn})J_m'(ra)-\rho_{nm}J_m(ra)J_m'(z_{mn})}{\sqrt{D_n^{(m)}}(\rho_{nm}^2-r^2)}
\\&=\frac{\rho_{nm}J_m(ra)J_m'(z_{nm})}{\sqrt{D_n^{(m)}}(r^2-\rho_{nm}^2)}
=\sqrt{2}(-1)^n\rho_{nm}\frac{J_m(ra)}{r^2-\rho_{nm}^2}.
\end{align*}
We then deduce that
\begin{equation}\label{JA.ge.alt00}
J_m(rs)=\sqrt{2}J_m(ar)\sum_{n=1}^\infty(-1)^n\rho_{nm}\frac{\mathcal{J}_{nm}^a(s)}{r^2-\rho_{nm}^2}.
\end{equation}
Applying Equation (\ref{JA.ge.alt00}) in (\ref{JA.e0}), we get
\begin{align*}
e^{\ii rs\cos(\alpha-\theta)}
&=\sum_{m=-\infty}^{\infty}\ii^m J_m(rs)e^{\ii m\theta}e^{-\ii m\alpha}
\\&=\sum_{m=-\infty}^{\infty}\ii^m\left(\sum_{n=1}^\infty(-1)^n\sqrt{2}\rho_{nm}\frac{J_m(ar)}{r^2-\rho_{nm}^2}\mathcal{J}_{nm}^a(s)\right)e^{\ii m\theta}e^{-\ii m\alpha}
\\&=\sqrt{2}\sum_{n=1}^\infty\sum_{m=-\infty}^{\infty}(-1)^n \rho_{nm}\frac{\ii^mJ_m(ar)}{r^2-\rho_{nm}^2}\mathcal{J}_{nm}^a(s)e^{\ii m\theta}e^{-\ii m\alpha}
\\&=2\sqrt{\pi}\sum_{m=-\infty}^{\infty}\sum_{n=1}^{\infty}
(-1)^n\rho_{nm}\frac{\ii^mJ_m(ar)}{r^2-\rho_{nm}^2}e^{-\ii m\alpha}\Psi_{nm}^a(s,\theta).
\end{align*}
\end{proof}

For $\mathbf{x}:=(x_1,x_2)^T\in\mathbb{R}^2$, let
\[
\rho(\mathbf{x})=\rho(x_1,x_2):=\sqrt{x_1^2+x_2^2},
\]
and
$0\le \Phi(\mathbf{x})=\Phi(x_1,x_2)<2\pi$ be given by
\[
x_1=\rho(x_1,x_2)\cos\Phi(x_1,x_2),\hspace{1cm}x_2=\rho(x_1,x_2)\sin\Phi(x_1,x_2).
\]
We may denote $\rho(\mathbf{x})$ with $|\mathbf{x}|$ as well.
\begin{corollary}
{\it Let $a>0$, $\mathbf{k}\in\mathbb{Z}^2$, $n\in\mathbb{N}$, and $m\in\mathbb{Z}$. We then have
\begin{equation}\label{main.alt.ZBC}
\int_0^a\int_0^{2\pi}e^{\pi\ii a^{-1}s\mathbf{u}_\theta^T\mathbf{k}}\overline{\Psi_{nm}^a(s\mathbf{u}_\theta)}sdsd\theta=2a\sqrt{\pi}(-1)^{n}\ii^{m}z_{nm}\frac{J_{m}(\pi|\mathbf{k}|)e^{-\ii m\Phi(\mathbf{k})}}{\pi^2|\mathbf{k}|^2-z_{nm}^2}.
\end{equation}
}\end{corollary}
\begin{proof}
Let $a>0$, $\mathbf{k}\in\mathbb{Z}^2$, $n\in\mathbb{N}$, and $m\in\mathbb{Z}$.
Applying (\ref{JA.ge}), for $r:=\pi a^{-1}|\mathbf{k}|$ and $\alpha:=\Phi(\mathbf{k})$, we get
\begin{align*}
\int_0^a\int_0^{2\pi}e^{\pi\ii a^{-1}s\mathbf{u}_\theta^T\mathbf{k}}\overline{\Psi_{nm}^a(s\mathbf{u}_\theta)}sdsd\theta
&=\int_0^a\int_0^{2\pi}e^{\pi\ii a^{-1}s|\mathbf{k}|\cos(\Phi(\mathbf{k})-\theta)}\overline{\Psi_{nm}^a(s\mathbf{u}_\theta)}sdsd\theta
\\&=\int_0^a\int_0^{2\pi}e^{\ii rs\cos(\alpha-\theta)}\overline{\Psi_{nm}^a(s\mathbf{u}_\theta)}sdsd\theta
\\&=2\sqrt{\pi}(-1)^{n}\ii^{m}\rho_{nm}\frac{\ii^mJ_m(ar)}{r^2-\rho_{nm}^2}e^{-\ii m\alpha}
\\&=2\sqrt{\pi}(-1)^{n}\ii^{m}\rho_{nm}\frac{J_{m}(\pi|\mathbf{k}|)e^{-\ii m\Phi(\mathbf{k})}}{\pi^2a^{-2}|\mathbf{k}|^2-\rho_{nm}^2}
\\&=2\sqrt{\pi}(-1)^{n}\ii^{m}a^{-1}z_{nm}\frac{J_{m}(\pi|\mathbf{k}|)e^{-\ii m\Phi(\mathbf{k})}}{\pi^2a^{-2}|\mathbf{k}|^2-a^{-2}z_{nm}^2}
\\&=2a\sqrt{\pi}(-1)^{n}\ii^{m}z_{nm}\frac{J_{m}(\pi|\mathbf{k}|)e^{-\ii m\Phi(\mathbf{k})}}{\pi^2|\mathbf{k}|^2-z_{nm}^2}.
\end{align*}
\end{proof}
Next result presents a closed form for coefficients of square integrable functions on disks, with respect to zero-valued boundary condition.
\begin{theorem}\label{TH.C.ZBVC}
Let $a>0$ and $\xi\in L^2(\mathbb{B}_a^2)$ be a function. Also, let $m\in\mathbb{Z}$ and $n\in\mathbb{N}$. We then have
\begin{equation}\label{C.ZBVC}
C_{n,m}^a(\xi)=\sum_{\mathbf{k}\in\mathbb{Z}^2}c_a(\mathbf{k};n,m)\widehat{\xi}
\{\mathbf{k}\},
\end{equation}
with
\begin{equation}\label{C.ZBVC.alt}
\widehat{\xi}\{\mathbf{k}\}:=\int_0^a\int_0^{2\pi}\xi(r,\theta)e^{-\pi\ii a^{-1}r(k_1\cos\theta+k_2\sin\theta)}rdrd\theta,\hspace{1cm}{\rm for\ all}\ \mathbf{k}=(k_1,k_2)^T\in\mathbb{Z}^2
\end{equation}
and
\begin{equation}
c_a(\mathbf{k};n,m):=\sqrt{\pi}(-1)^{n}\ii^{m}\rho_{nm}\frac{J_{m}(\pi|\mathbf{k}|)e^{-\ii m\Phi(\mathbf{k})}}{2(\pi^2|\mathbf{k}|^2-z_{nm}^2)},\hspace{1cm}{\rm for\ all}\ \mathbf{k}=(k_1,k_2)^T\in\mathbb{Z}^2.
\end{equation}
\end{theorem}
\begin{proof}
Let $a>0$ and $\Omega_a:=[-a,a]^2$. Let $\xi\in L^2(\mathbb{B}_a^2)$ be a function and 
$u:=E(\xi)$ be the canonical extension of $\xi$ to the rectangle $\Omega_a$ by zero-padding. Then $u=E(\xi)$ is supported in $\mathbb{B}_a^2$. Also, we have $u=E(\xi)\in L^2(\Omega_a)$ with 
\[
\|E(\xi)\|_{L^2(\Omega_a)}=\|\xi\|_{L^2(\mathbb{B}_a^2)}.
\]
Hence, using the classical Fourier series of the function $u=E(\xi)$, we can write 
\begin{equation}
u(\mathbf{x})
=\frac{1}{4a^2}\sum_{\mathbf{k}\in\mathbb{Z}^2}\widehat{u}(\mathbf{k})e^{\pi\ii a^{-1}\mathbf{x}^T\mathbf{k}},\hspace{1cm}{\rm for\ all}\ \mathbf{x}=(x_1,x_2)^T\in\Omega_a,
\end{equation}
where for the integral vector 
$\mathbf{k}:=(k_1,k_2)\in\mathbb{Z}^2$, we have
\[
\widehat{u}(\mathbf{k})=\int_{-a}^a\int_{-a}^au(x,y)e^{-\pi\ii a^{-1}(k_1x+k_2y)}dxdy.
\]
Since $u=E(\xi)$ is supported in the disk $\mathbb{B}_a^2$, we have 
\begin{align*}
\widehat{u}(\mathbf{k})
&=\int_{\Omega_a}E(\xi)(x,y)e^{-\pi\ii a^{-1}(k_1x+k_2y)}dxdy
\\&=\int_{\mathbb{B}_a^2}\xi(x,y)e^{-\pi\ii a^{-1}(k_1x+k_2y)}dxdy
\\&=\int_0^a\int_0^{2\pi}\xi(r,\theta)e^{-\pi\ii a^{-1}(k_1r\cos\theta+k_2r\sin\theta)}rdrd\theta
\\&=\int_0^a\int_0^{2\pi}\xi(r,\theta)e^{-\pi\ii a^{-1}r(k_1\cos\theta+k_2\sin\theta)}rdrd\theta=\widehat{\xi}\{\mathbf{k}\}.
\end{align*}
Let $0<s\le a$ and $0\le\theta\le2\pi$. Therefore, for $\mathbf{x}:=s\mathbf{u}_\theta$, we get 
\begin{align*}
\xi(s\mathbf{u}_\theta)&=E(\xi)(\mathbf{x})
\\&=\frac{1}{4a^2}\sum_{\mathbf{k}\in\mathbb{Z}^2}\widehat{E(\xi)}(\mathbf{k})e^{\pi\ii a^{-1}\mathbf{x}^T\mathbf{k}}
\\&=\frac{1}{4a^2}\sum_{\mathbf{k}\in\mathbb{Z}^2}\widehat{\xi}\{\mathbf{k}\}e^{\pi\ii a^{-1}\mathbf{x}^T\mathbf{k}}
=\frac{1}{4a^2}\sum_{\mathbf{k}\in\mathbb{Z}^2}\widehat{\xi}\{\mathbf{k}\}e^{\pi\ii a^{-1}s\mathbf{u}_\theta^T\mathbf{k}}.
\end{align*}
Hence, using Equation (\ref{main.alt.ZBC}), we achieve 
\begin{align*}
C_{n,m}^a(\xi)&=\int_0^a\int_0^{2\pi}\xi(s\mathbf{u}_\theta)\overline{\Psi_{nm}^a(s\mathbf{u}_\theta)}sdsd\theta
\\&=\int_0^a\int_0^{2\pi}\left(\frac{1}{4a^2}\sum_{\mathbf{k}\in\mathbb{Z}^2}\widehat{\xi}\{\mathbf{k}\}e^{\pi\ii a^{-1}s\mathbf{u}_\theta^T\mathbf{k}}\right)\overline{\Psi_{nm}^a(s\mathbf{u}_\theta)}sds d\theta
\\&=\frac{1}{4a^2}\sum_{\mathbf{k}\in\mathbb{Z}^2}\widehat{\xi}\{\mathbf{k}\}\left(\int_0^a\int_0^{2\pi}e^{\pi\ii a^{-1}s\mathbf{u}_\theta^T\mathbf{k}}\overline{\Psi_{nm}^a(s\mathbf{u}_\theta)}sdsd\theta\right)
=\sum_{\mathbf{k}\in\mathbb{Z}^2}c_a(\mathbf{k};n,m)\widehat{\xi}\{\mathbf{k}\}.
\end{align*}
\end{proof}

\begin{corollary}
{\it Let $a>0$ and $\xi\in L^2(\mathbb{B}_a^2)$ be a function. We then have 
\[
\xi(r,\theta)=\sum_{m=-\infty}^{+\infty}\sum_{n=1}^\infty C_{n,m}^a(\xi)\Psi_{nm}^a(r,\theta),\hspace{0.5cm}{\rm for\ a.e.}\ 0\le r\le a,\ 0\le \theta\le 2\pi.
\]
In particular, if $\xi:\mathbb{B}_a^2\to\mathbb{C}$ is continuous, we have 
\[
\xi(r,\theta)=\sum_{m=-\infty}^{+\infty}\sum_{n=1}^\infty C_{n,m}^a(\xi)\Psi_{nm}^a(r,\theta),\hspace{0.5cm}{\rm for\ all}\ 0\le r\le a,\ 0\le \theta\le 2\pi.
\]
}\end{corollary}

\begin{remark}
The equation (\ref{C.ZBVC}) guarantees that the Fourier-Bessel coefficients of functions supported in disks can be computed from the standard Fourier coefficients $\widehat{\xi}\{\mathbf{k}\}$, which can be implemented by FFT.
\end{remark}

Next result gives an explicit closed form for coefficients of zero-padded functions on disks.

\begin{theorem}\label{g.zp.ZBVC}
{\it Let $a>0$ and $f:\mathbb{R}^2\to\mathbb{C}$ be a square-integrable function 
on the disk $\mathbb{B}_a^2$. We then have
\[
f(r,\theta)=\sum_{m=-\infty}^{+\infty}\sum_{n=1}^\infty C_{n,m}^a[f]\Psi_{nm}^a(r,\theta),\hspace{0.5cm}{\rm for\ a.e.}\ 0\le r\le a,\ 0\le \theta\le 2\pi,
\]
where
\begin{equation}\label{CnmZf.ZBVC}
C_{n,m}^a[f]:=\sum_{\mathbf{k}\in\mathbb{Z}^2}c_a(\mathbf{k};n,m)\widehat{f}[a^{-1}\mathbf{k};a],
\hspace{1cm}{\rm for\ all}\ m\in\mathbb{Z},\ n\in\mathbb{N},
\end{equation}
with 
\begin{equation}
\widehat{f}[\bom;a]:=\int_0^a\int_0^{2\pi}f(r,\theta)e^{-\pi\ii (\omega_1\cos\theta+\omega_2\sin\theta)}rdrd\theta,\hspace{1cm}{\rm for\ all}\ \bom=(\omega_1,\omega_2)^T\in\mathbb{R}^2.
\end{equation}
}\end{theorem}
\begin{proof}
Let $a>0$ and $f:\mathbb{R}^2\to\mathbb{C}$ be a function such that 
restriction of $f$ to the disk $\mathbb{B}_a^2$ is square-integrable on 
$\mathbb{B}_a^2$. Let $\xi:=R(f)$
be the restriction of $f$ to the disk $\mathbb{B}_a^2$.
We then have $\xi\in L^2(\mathbb{B}_a^2)$. Hence, we have 
\begin{equation}\label{g.zp.ZBVC.alt0}
\xi=\sum_{m=-\infty}^{+\infty}\sum_{n=1}^\infty C_{n,m}^a(\xi)\Psi_{nm}^a.
\end{equation}
Suppose $m\in\mathbb{Z}$ and $n\in\mathbb{N}$. Let $C_{n,m}[f]:=C_{n,m}(\xi)$. 
Hence, using Equation (\ref{C.ZBVC}) for the function $\xi\in L^2(\mathbb{B}_a^2)$, we can write 
\begin{align*}
C_{n,m}^a[f]
=\sum_{\mathbf{k}\in\mathbb{Z}^2}c_a(\mathbf{k};n,m)\widehat{\xi}\{\mathbf{k}\}.
\end{align*}
Let $\mathbf{k}:=(k_1,k_2)^T\in\mathbb{Z}^2$ be given. Invoking Equation (\ref{C.ZBVC.alt}), we get 
\begin{align*}
\widehat{\xi}\{\mathbf{k}\}&=\int_0^a\int_0^{2\pi}\xi(r,\theta)e^{-\pi\ii a^{-1}r(k_1\cos\theta+k_2\sin\theta)}rdrd\theta
\\&=\int_0^a\int_0^{2\pi}R(f)(r,\theta)e^{-\pi\ii a^{-1}r(k_1\cos\theta+k_2\sin\theta)}rdrd\theta
\\&=\int_0^a\int_0^{2\pi}f(r,\theta)e^{-\pi\ii a^{-1}r(k_1\cos\theta+k_2\sin\theta)}
rdrd\theta=\widehat{f}[a^{-1}\mathbf{k};a],
\end{align*}
which implies that 
\begin{equation}\label{g.zp.ZBVC.alt00}
C_{n,m}^a[f]
=\sum_{\mathbf{k}\in\mathbb{Z}^2}c_a(\mathbf{k};n,m)\widehat{f}[a^{-1}\mathbf{k};a].
\end{equation}
Applying Equation (\ref{g.zp.ZBVC.alt00}) in (\ref{g.zp.ZBVC.alt0}) completes the proof.
\end{proof}

\begin{corollary}
{\it Let $a>0$ and $f\in L^2(\mathbb{R}^2)$ be a function. We then have 
\[
f(r,\theta)=\sum_{m=-\infty}^{+\infty}\sum_{n=1}^\infty C_{n,m}^a[f]\Psi_{nm}^a(r,\theta),\hspace{0.5cm}{\rm for\ a.e.}\ 0\le r\le a,\ 0\le \theta\le 2\pi,
\]
}\end{corollary}
\begin{proof}
Let $a>0$ and $f\in L^2(\mathbb{R}^2)$ be a function. Let $\xi:=R(f)$
be the restriction of $f$ to the disk $\mathbb{B}_a^2$.
We then have $\xi\in L^2(\mathbb{B}_a^2)$ with $\|\xi\|_{L^2(\mathbb{B}_a^2)}\le\|f\|_{L^2(\mathbb{R}^2)}$. Indeed, we have 
\[
\|\xi\|_{L^2(\mathbb{B}_a^2)}^2=\int_{\mathbb{B}_a^2}|f(\mathbf{x})|^2d\mathbf{x}
\le\int_{\mathbb{R}^2}|f(\mathbf{x})|^2d\mathbf{x}=\|f\|_{L^2(\mathbb{R}^2)}^2.
\]
Thus, $f$ is square-integrable on $\mathbb{B}_a^2$. Applying Theorem \ref{g.zp.ZBVC}, for the function $f$, completes the proof. 
\end{proof}

We then conclude the following results concerning continuous functions. 

\begin{proposition}\label{gr.ZBVC.pw}
{\it Let $a>0$ and $f:\mathbb{R}^2\to\mathbb{C}$ be a continuous function. 
We then have
\[
f(r,\theta)=\sum_{m=-\infty}^{+\infty}\sum_{n=1}^\infty C_{n,m}^a[f]\Psi_{nm}^a(r,\theta),\hspace{0.5cm}{\rm for\ all}\ 0\le r\le a,\ 0\le \theta\le 2\pi.
\]
}\end{proposition}
\begin{proof}
Let $a>0$ and $f:\mathbb{R}^2\to\mathbb{C}$ be a continuous function. Since
$\mathbb{B}_a^2$ is compact in $\mathbb{R}^2$ and hence of finite Lebesgue measure, we deduce that $f$ is square-integrable on $\mathbb{B}_a^2$. Using Theorem \ref{g.zp.ZBVC}, for the function $f$, we get 
\[
f(r,\theta)=\sum_{m=-\infty}^{+\infty}\sum_{n=1}^\infty C_{n,m}^a[f]\Psi_{nm}^a(r,\theta),\hspace{0.5cm}{\rm for\ a.e.}\ 0\le r\le a,\ 0\le \theta\le 2\pi.
\]
Then continuity of $f$ implies the point-wise convergence of the series, which completes the proof.
\end{proof}

\begin{corollary}
{\it Let $a>0$ and $f:\mathbb{R}^2\to\mathbb{C}$ be a continuous function supported in 
$\mathbb{B}_a^2$. We then have 
\[
f(r,\theta)=\sum_{m=-\infty}^{+\infty}\sum_{n=1}^\infty C_{n,m}^a[f]\Psi_{nm}^a(r,\theta),\hspace{0.5cm}{\rm for\ all}\ 0\le r\le a,\ 0\le \theta\le 2\pi,
\]
where
\begin{equation}
C_{n,m}^a[f]:=\sum_{\mathbf{k}\in\mathbb{Z}^2}c_a(\mathbf{k};n,m)\widehat{f}(a^{-1}\mathbf{k}),
\hspace{1cm}{\rm for\ all}\ m\in\mathbb{Z},\ n\in\mathbb{N}.
\end{equation}
}\end{corollary}
\begin{proof}
Let $a>0$ and $f:\mathbb{R}^2\to\mathbb{C}$ be a continuous function supported in 
$\mathbb{B}_a^2$. Since $f$ is continuous, using Proposition \ref{gr.ZBVC.pw}, conclude that 
\[
f(r,\theta)=\sum_{m=-\infty}^{+\infty}\sum_{n=1}^\infty C_{n,m}^a[f]\Psi_{nm}^a(r,\theta),\hspace{0.5cm}{\rm for\ all}\ 0\le r\le a,\ 0\le \theta\le 2\pi,
\]
where
\begin{equation}
C_{n,m}^a[f]:=\sum_{\mathbf{k}\in\mathbb{Z}^2}c_a(\mathbf{k};n,m)\widehat{f}[a^{-1}\mathbf{k};a],
\hspace{1cm}{\rm for\ all}\ m\in\mathbb{Z},\ n\in\mathbb{N}.
\end{equation}
Let $\mathbf{k}:=(k_1,k_2)^T\in\mathbb{Z}^2$ be an integral vector. 
Because $f$ is supported in disk $\mathbb{B}_a^2$, we can write 
\begin{align*}
\widehat{f}[a^{-1}\mathbf{k};a]&=\int_0^a\int_0^{2\pi}f(r,\theta)e^{-\pi\ii a^{-1}r(k_1\cos\theta+k_2\sin\theta)}rdrd\theta
\\&=\int_{\mathbb{B}_a^2}f(\mathbf{x})e^{-\pi\ii a^{-1}\mathbf{x}^T\mathbf{k}}d\mathbf{x}
\\&=\int_{\mathbb{R}^2}f(\mathbf{x})e^{-\pi\ii a^{-1}\mathbf{x}^T\mathbf{k}}d\mathbf{x}=\widehat{f}(a^{-1}\mathbf{k}),
\end{align*}
which completes the proof. 
\end{proof}

Let $\mathcal{R}:=\left\{\rho(k_1,k_2):k_1,k_2\in\mathbb{Z}\right\}$.
For each $r\in\mathcal{R}$, let
\[
\Theta_r:=\left\{\Phi(i,j):r=\rho(i,j),~~i,j\in\mathbb{Z}\right\}.
\]

\begin{proposition}
{\it With above assumptions we have
\begin{enumerate}
\item $\mathbb{N}\cup\{0\}\subseteq\mathcal{R}\subseteq\sqrt{\mathbb{N}}:=\{\sqrt{n}:n\in\mathbb{N}\cup\{0\}\}$.
\item $\mathcal{R}$ is a discrete subset of $[0,\infty)$.
\item For each $r\in\mathcal{R}$, the set $\Theta_r$ is a finite subset of $[0,2\pi)$.
\item $\mathbb{Z}^2=\bigcup_{r\in\mathcal{R}}\{(r\cos\theta,r\sin\theta)^T:\theta\in\Theta_r\}$.
\end{enumerate}
}\end{proposition}
\begin{proof}
(1)-(3) are straightforward.

(4) Let $\mathbf{x}\in\bigcup_{r\in\mathcal{R}}\{(r\cos\theta,r\sin\theta)^T:\theta\in\Theta_r\}$. Suppose $r\in\mathcal{R}$ and $\theta\in\Theta_r$ with $\mathbf{x}=(r\cos\theta,r\sin\theta)^T$. Hence, $\theta=\Phi(i,j)$ with $\rho(i,j)=r$, for some $i,j\in\mathbb{Z}$. We then have
\[
r\cos\theta=\rho(i,j)\cos\Phi(i,j)=\sqrt{i^2+j^2}\frac{i}{\sqrt{i^2+j^2}}=i\in\mathbb{Z},
\]
and
\[
r\sin\theta=\rho(i,j)\sin\Phi(i,j)=\sqrt{i^2+j^2}\frac{j}{\sqrt{i^2+j^2}}=j\in\mathbb{Z}.
\]
Thus, we deduce that $\mathbf{x}=(r\cos\theta,r\sin\theta)^T\in\mathbb{Z}^2$.
Therefore, we get $\bigcup_{r\in\mathcal{R}}\{(r\cos\theta,r\sin\theta)^T:\theta\in\Theta_r\}\subseteq\mathbb{Z}^2$. Conversely, let $\mathbf{x}=(k_1,k_2)\in\mathbb{Z}^2$ be given.
We then have $k_1,k_2\in\mathbb{Z}$ and hence we get
$k_1=\rho(k_1,k_2)\cos\Phi(k_1,k_2)$, and $k_2=\rho(k_1,k_2)\sin\Phi(k_1,k_2)$.
Then, we conclude that $\mathbf{x}=(r\cos\theta,r\sin\theta)^T$,
with $r:=\rho(k_1,k_2)$ and $\theta:=\Phi(k_1,k_2)$. This implies that $\mathbf{x}\in\bigcup_{r\in\mathcal{R}}\{(r\cos\theta,r\sin\theta)^T:\theta\in\Theta_r\}$ and hence
$\mathbb{Z}^2\subseteq\bigcup_{r\in\mathcal{R}}\{(r\cos\theta,r\sin\theta)^T:\theta\in\Theta_r\}$.
\end{proof}

We then present the following polarized version of Theorem \ref{TH.C.ZBVC}.

\begin{theorem}
Let $a>0$ and $\xi\in L^2(\mathbb{B}_a^2)$ be a function. Also, let $m\in\mathbb{Z}$ and $n\in\mathbb{N}$. We then have
\begin{equation}\label{C.ZBVC.Polar}
C_{n,m}^a(\xi)=\sum_{\tau\in\mathcal{R}}\sum_{\alpha\in\Phi_\tau}A_{mn}^a(\tau,\alpha)\widehat{\xi}\{\tau\mathbf{u}_\alpha\},
\end{equation}
where
\begin{equation}
A_{mn}^a(\tau,\alpha):=\sqrt{\pi}(-1)^{n}\ii^{m}\rho_{nm}\frac{J_{m}(\pi\tau)e^{-\ii m\alpha}}{2(\pi^2\tau^2-z_{nm}^2)}.
\end{equation}
\end{theorem}
\begin{proof}
Let $m\in\mathbb{Z}$ and $n\in\mathbb{N}$. First, suppose that $\tau\in\mathcal{R}$ and $\alpha\in\Phi_\tau$. Let $\mathbf{k}:=\tau\mathbf{u}_{\alpha}=(\tau\cos\alpha,\tau\sin\alpha)^T\in\mathbb{Z}^2$. Thus, $|\mathbf{k}|=\tau$ and $\Phi(\mathbf{k})=\alpha$. We then have
\begin{align*}
c_a(\tau\mathbf{u}_\alpha;n,m)&=c_a(\mathbf{k};n,m)
\\&=\sqrt{\pi}(-1)^{n}\ii^{m}\rho_{nm}\frac{J_{m}(\pi|\mathbf{k}|)e^{-\ii m\Phi(\mathbf{k})}}{2(\pi^2|\mathbf{k}|^2-z_{nm}^2)}
=\sqrt{\pi}(-1)^{n}\ii^{m}\rho_{nm}\frac{J_{m}(\pi\tau)e^{-\ii m\alpha}}{2(\pi^2\tau^2-z_{nm}^2)}.
\end{align*}
Therefore, using (\ref{C.ZBVC}), we get
\begin{align*}
C_{n,m}^a(\xi)&=\sum_{\mathbf{k}\in\mathbb{Z}^2}c_a(\mathbf{k};n,m)\widehat{\xi}\{\mathbf{k}\}
\\&=\sum_{\tau\in\mathcal{R}}\sum_{\alpha\in\Phi_\tau}c_a(\tau\mathbf{u}_\alpha;n,m)\widehat{\xi}\{\tau\mathbf{u}_\alpha\}
\\&=\sqrt{\pi}(-1)^{n}\ii^{m}\rho_{nm}\sum_{\tau\in\mathcal{R}}\sum_{\alpha\in\Phi_\tau}\frac{J_{m}(\pi\tau)e^{-\ii m\alpha}}{2(\pi^2\tau^2-z_{nm}^2)}\widehat{\xi}\{\tau\mathbf{u}_\alpha\}
=\sum_{\tau\in\mathcal{R}}\sum_{\alpha\in\Phi_\tau}A_{mn}^a(\tau,\alpha)\widehat{\xi}\{\tau\mathbf{u}_\alpha\}.
\end{align*}
\end{proof}

\begin{corollary}
{\it Let $a>0$ and $\xi:\mathbb{B}_a^2\to\mathbb{C}$ be a continuous function. We then have
\[
\xi(r,\theta)=\sum_{m=-\infty}^{+\infty}\sum_{n=1}^\infty C_{n,m}^a(\xi)\Psi_{nm}^a(r,\theta),\hspace{0.5cm}{\rm for\ all}\ 0\le r\le a,\ 0\le \theta\le 2\pi
\]
where
\begin{equation}
C_{n,m}^a(\xi)=\sum_{\tau\in\mathcal{R}}\sum_{\alpha\in\Phi_\tau}A_{mn}^a(\tau,\alpha)\widehat{\xi}\{\tau\mathbf{u}_\alpha\}.
\end{equation}
}\end{corollary}

Next result gives a polarized version for explicit closed form of coefficients for zero-padded functions.

\begin{proposition}
{\it Let $a>0$ and $f:\mathbb{R}^2\to\mathbb{C}$ be a square-integrable function 
on the disk $\mathbb{B}_a^2$. We then have
\[
f(r,\theta)=\sum_{m=-\infty}^{+\infty}\sum_{n=1}^\infty C_{n,m}^a[f]\Psi_{nm}^a(r,\theta),
\hspace{0.5cm}{\rm for\ a.e.}\ 0\le r\le a,\ 0\le \theta\le 2\pi,\]
with
\begin{equation}
C_{n,m}^a[f]=\sum_{\tau\in\mathcal{R}}\sum_{\alpha\in\Phi_\tau}A_{mn}^a(\tau,\alpha)\widehat{f}[a^{-1}\tau\mathbf{u}_\alpha;a].
\end{equation}
}\end{proposition}

\begin{corollary}
{\it Let $a>0$ and $f:\mathbb{R}^2\to\mathbb{C}$ be a continuous function. 
We then have
\[
f(r,\theta)=\sum_{m=-\infty}^{+\infty}\sum_{n=1}^\infty C_{n,m}^a[f]\Psi_{nm}^a(r,\theta),
\hspace{0.5cm}{\rm for\ all}\ 0\le r\le a,\ 0\le \theta\le 2\pi,
\]
}\end{corollary}

\subsection{Zero-padded convolutions on disks using zero-valued boundary condition}
We then continue by investigating analytical aspects of Fourier-Bessel approximations for zero-padded convolutions and hence convolution of functions supported in disks.

The following theorem introduces a constructive method for computing the Fourier-Bessel coefficients of zero-padded convolutions of functions on disks.

\begin{theorem}
Let $a>0$ and $b:=a/2$. Suppose $f_j:\mathbb{R}^2\to\mathbb{C}$ with $j\in\{1,2\}$ are  
square integrable functions on $\mathbb{B}_b^2$. Let $m\in\mathbb{Z}$ and $n\in\mathbb{N}$. We then have
\begin{equation}\label{conv.ZBC}
C_{n,m}^a[f_1\circledast f_2]=\sum_{\mathbf{k}\in\mathbb{Z}^2}c_a(\mathbf{k};n,m)\widehat{f_1}[a^{-1}\mathbf{k};b]\widehat{f_2}[a^{-1}\mathbf{k};b],
\end{equation}
where
\begin{equation}\label{zp.hat}
\widehat{f_j}[\bom;b]:=\int_0^b\int_0^{2\pi}f_j(r,\theta)e^{-\pi\ii r(\omega_1\cos\theta+\omega_2\sin\theta)}rdrd\theta,\hspace{1cm}{\rm for\ all}\ \bom:=(\omega_1,\omega_2)^T\in\mathbb{R}^2.
\end{equation}
\end{theorem}
\begin{proof}
Let $f_j:\mathbb{R}^2\to\mathbb{C}$ with $j\in\{1,2\}$ be square integrable functions on $\mathbb{B}_b^2$. Let $\xi_j:=R(f_j)$ be the restriction of $f_j$ to the disk 
$\mathbb{B}_b^2$ and $E(\xi_j)$ be the extension of $\xi_j$ to $\mathbb{R}^2$ by zero-padding.  Thus, we get $\xi_j\in L^2(\mathbb{B}_b^2)$ and hence $E(\xi_j)\in L^2(\mathbb{R}^2)$. Since $L^2(\mathbb{B}_b^2)\subseteq L^1(\mathbb{B}_b^2)$, we get $\xi_j\in L^1(\mathbb{B}_b^2)$ as well. Therefore, we have $E(\xi_j)\in L^1\cap L^2(\mathbb{R}^2)$. By definition of zero-padded convolutions, we have  
\[
f_1\circledast f_2=E(\xi_1)\ast E(\xi_2).
\]
Using Proposition \ref{222.R2}, we deduce that $f:=f_1\circledast f_2$ is square integrable on $\mathbb{B}_a^2$. Let $\mathbf{k}\in\mathbb{Z}^2$ be given. Since 
$f=f_1\circledast f_2$ is supported in the disk 
$\mathbb{B}_a^2$, we have  
\[
\widehat{f}[a^{-1}\mathbf{k};a]=\widehat{f}(a^{-1}\mathbf{k}).
\]
By the convolution property of Fourier transform, we get 
\begin{equation}\label{conv.zp.0}
\widehat{f}[a^{-1}\mathbf{k};a]=\widehat{E(\xi_1)}(a^{-1}\mathbf{k})\widehat{E(\xi_2)}(a^{-1}\mathbf{k}).
\end{equation}
Since each $E(\xi_j)$ is supported in the disk $\mathbb{B}_b^2$, we get 
\begin{equation}\label{conv.zp.00}
\widehat{E(\xi_j)}(a^{-1}\mathbf{k})=\widehat{f_j}[a^{-1}\mathbf{k};b].
\end{equation}
Indeed, we can write 
\begin{align*}
\widehat{E(\xi_j)}(a^{-1}\mathbf{k})
&=\int_{\mathbb{R}^2}E(\xi_j)(\mathbf{x})e^{-\ii a^{-1}\mathbf{x}^T\mathbf{k}}d\mathbf{x}
\\&=\int_{\mathbb{B}_b^2}\xi_j(\mathbf{x})e^{-\ii a^{-1}\mathbf{x}^T\mathbf{k}}d\mathbf{x}
\\&=\int_0^b\int_0^{2\pi}f_j(r,\theta)e^{-\pi\ii a^{-1}r(k_1\cos\theta+k_2\sin\theta)}rdrd\theta=\widehat{f_j}[a^{-1}\mathbf{k};b].
\end{align*}
Applying Equation (\ref{conv.zp.00}) in Equation (\ref{conv.zp.0}), we get 
\begin{equation}\label{conv.zp.000}
\widehat{f}[a^{-1}\mathbf{k};a]=\widehat{f_1}[a^{-1}\mathbf{k};b]\widehat{f_2}[a^{-1}\mathbf{k};b].
\end{equation}
Let $m\in\mathbb{Z}$ and $n\in\mathbb{N}$.
Then, using Equation (\ref{conv.zp.000}) in Equation (\ref{CnmZf.ZBVC}), we get
\begin{align*}
C_{n,m}^a[f_1\circledast f_2]&=C_{n,m}^a[f]
\\&=\sum_{\mathbf{k}\in\mathbb{Z}^2}c_a(\mathbf{k};n,m)\widehat{f}[a^{-1}\mathbf{k};a]
\\&=\sum_{\mathbf{k}\in\mathbb{Z}^2}c_a(\mathbf{k};n,m)\widehat{f_1}[a^{-1}\mathbf{k};b]\widehat{f_2}[a^{-1}\mathbf{k};b],
\end{align*}
which completes the proof.
\end{proof}

\begin{corollary}
{\it Let $a>0$ and $b:=a/2$. Suppose $f_j:\mathbb{R}^2\to\mathbb{C}$ with $j\in\{1,2\}$ are functions integrable on $\mathbb{B}_b^2$. We then have
\begin{equation}\label{2D.EX.Fr.ast.zerop}
(f_1\circledast f_2)(r,\theta)=\sum_{m=-\infty}^\infty\sum_{n=1}^\infty C_{n,m}^{a}[f_1\circledast f_2]\Psi_{nm}^a(r,\theta),\hspace{0.5cm}{\rm for\ a.e.}\ 0\le r\le a,\ 0\le \theta\le 2\pi,
\end{equation}
where
\begin{equation}
C_{n,m}^a[f_1\circledast f_2]=\sum_{\mathbf{k}\in\mathbb{Z}^2}c_a(\mathbf{k};n,m)\widehat{f_1}[a^{-1}\mathbf{k};b]\widehat{f_2}[a^{-1}\mathbf{k};b].
\end{equation}
}\end{corollary}

\begin{proposition}
{\it Let $a>0$ and $b:=a/2$. Suppose $f_j:\mathbb{R}^2\to\mathbb{C}$ with $j\in\{1,2\}$ 
are continuous functions supported in $\mathbb{B}^2_{b}$. We then have
\begin{equation}\label{2D.EX.Fr.ast.ZBVC.cts}
(f_1\ast f_2)(r,\theta)=\sum_{m=-\infty}^\infty\sum_{n=1}^{\infty} C_{n,m}^{a}[f_1\ast f_2]\Psi_{nm}^a(r,\theta),\hspace{0.5cm}{\rm for\ all}\ 0\le r\le a,\ 0\le \theta\le 2\pi
\end{equation}
where
\begin{equation}\label{ZBC.cts.conv}
C_{n,m}^a[f_1\ast f_2]=\sum_{\mathbf{k}\in\mathbb{Z}^2}c_a(\mathbf{k};n,m)\widehat{f_1}(a^{-1}\mathbf{k})\widehat{f_2}(a^{-1}\mathbf{k}).
\end{equation}
}\end{proposition}

We then present the following polarized version of closed forms for Fourier-Bessel coefficients of zero-padded convolutions on disks.

\begin{theorem}
Let $a>0$ and $b:=a/2$. Suppose $f_j:\mathbb{R}^2\to\mathbb{C}$ with $j\in\{1,2\}$ are continuous functions. Also, let $m\in\mathbb{Z}$ and $n\in\mathbb{N}$. We then have
\begin{equation}
C_{n,m}^a[f_1\circledast f_2]=\sum_{\tau\in\mathcal{R}}\sum_{\alpha\in\Phi_\tau}A_{mn}^a(\tau,\alpha)\widehat{f_1}[a^{-1}\tau\mathbf{u}_\alpha;b]\widehat{f_2}[a^{-1}\tau\mathbf{u}_\alpha;b].
\end{equation}
In particular, if $f_j:\mathbb{R}^2\to\mathbb{C}$ with $j\in\{1,2\}$ are continuous functions supported in $\mathbb{B}^2_{b}$, we have
\begin{equation}
C_{n,m}^a[f_1\circledast f_2]=\sum_{\tau\in\mathcal{R}}\sum_{\alpha\in\Phi_\tau}A_{mn}^a(\tau,\alpha)\widehat{f_1}(a^{-1}\tau\mathbf{u}_\alpha)\widehat{f_2}(a^{-1}\tau\mathbf{u}_\alpha).
\end{equation}
\end{theorem}

\subsection{Convolution of basis elements on disks using zero-valued boundary condition}
Let $a>0$ and $b:=a/2$. Let $f_j:\mathbb{R}^2\to\mathbb{C}$ with $j\in\{1,2\}$ be 
functions square integrable on the disk $\mathbb{B}_{b}^2$ with the associated Fourier-Bessel coefficients
$\left\{C_{n,m}^{b}[f_j]:n\in\mathbb{N},m\in\mathbb{Z}\right\}$. Hence, we can write
\begin{equation}
\xi_j=\sum_{m=-\infty}^\infty\sum_{n=1}^\infty C_{n,m}^{b}[f_j]\Psi_{nm}^{b},
\end{equation}
where $\xi_j:=R(f_j)$ is the restriction of $f_j$ into the disk $\mathbb{B}_b^2$ and 
\[
C_{n,m}^{b}[f_j]=\sum_{\mathbf{k}\in\mathbb{Z}^2}c_a(\mathbf{k};n,m)\widehat{f_j}[a^{-1}\mathbf{k};b],
\]
for $m\in\mathbb{Z}$ and $n\in\mathbb{N}$.

Thus, we get 
\[
E(\xi_j)=\sum_{m=-\infty}^\infty\sum_{n=1}^\infty C_{n,m}^{b}[f_j]E(\Psi_{nm}^{b}).
\]
Using linearity of convolutions, as linear operators, we get
\begin{align*}
f_1\circledast f_2 
&=E(\xi_1)\ast E(\xi_2)
\\&=\left(\sum_{m=-\infty}^\infty\sum_{n=1}^\infty C_{n,m}^{b}[f_1]E(\Psi_{nm}^{b})\right)\ast\left(\sum_{m'=-\infty}^\infty\sum_{n'=1}^\infty C_{n',m'}^{b}[f_2]E(\Psi_{n'm'}^{b})\right)
\\&=\sum_{m=-\infty}^\infty\sum_{n=1}^\infty\sum_{m'=-\infty}^\infty\sum_{n'=1}^\infty C_{n,m}^{b}[f_1]C_{n',m'}^{b}[f_2]E(\Psi_{nm}^{b})\ast E(\Psi_{n'm'}^{b})
\\&=\sum_{m=-\infty}^\infty\sum_{n=1}^\infty\sum_{m'=-\infty}^\infty\sum_{n'=1}^\infty C_{n,m}^{b}[f_1]C_{n',m'}^{b}[f_2]\Psi_{nm}^b\circledast\Psi_{n'm'}^b.
\end{align*}

Thus, we deduce that convolution of circular drums (basis elements) can be viewed as pre-computed kernels.

\begin{proposition}
{\it Let $a>0$ and $b:=a/2$. Suppose $\bom:=(\omega_1,\omega_2)^T\in\mathbb{R}^2$, $m\in\mathbb{Z}$, and $n\in\mathbb{N}$. We then have
\begin{equation}\label{psi.nm.hat.omega.ZPVC}
\widehat{E(\Psi_{nm}^b)}[\bom;b]
=\sqrt{\frac{2\pi}{N_n^{(m)}(b)}}\ii^{-m}e^{\ii m\Phi(\bom)}\frac{z_{nm}J_m(\pi|\bom|b)J_m'(z_{nm})}{\pi^2|\bom|^2-b^{-2}z_{nm}^2}.
\end{equation}
}\end{proposition}
\begin{proof}
Let $a>0$ and $b:=a/2$. Suppose $\bom:=(\omega_1,\omega_2)^T\in\mathbb{R}^2$, $m\in\mathbb{Z}$, and $n\in\mathbb{N}$. By applying Equation (\ref{zp.hat}), we get
\begin{align*}
\widehat{E(\Psi_{nm}^b)}[\bom;b]
&=\int_0^b\int_0^{2\pi}\Psi_{nm}^b(r,\theta)e^{-\pi\ii r(\omega_1\cos\theta+\omega_2\sin\theta)}rdrd\theta
=\int_0^{2\pi}\int_0^{b}\Psi_{nm}^b(r,\theta)e^{-\pi\ii r|\bom|\mathbf{u}_{\Phi(\bom)}^T\mathbf{u}_\theta}rdrd\theta
\\&=\int_0^{2\pi}\int_0^{b}\Psi_{nm}^b(r,\theta)\left(\sum_{l=-\infty}^{\infty}\ii^{-l}
J_l(\pi r|\bom|)e^{-\ii l\theta}e^{\ii l\Phi(\bom)}\right)rdrd\theta
\\&=\frac{1}{\sqrt{2\pi N_n^{(m)}(b)}}\int_0^{2\pi}\int_0^{b} J_m(b^{-1}z_{nm}r)e^{\ii m\theta}\left(\sum_{l=-\infty}^{\infty}\ii^{-l}
J_l(\pi r|\bom|)e^{-\ii l\theta}e^{\ii l\Phi(\bom)}\right)rdrd\theta
\\&=\frac{1}{\sqrt{2\pi N_n^{(m)}(b)}}\sum_{l=-\infty}^{\infty}\ii^{-l}e^{\ii l\Phi(\bom)}
\left(\int_0^{2\pi}\int_0^{b} J_m(b^{-1}z_{nm}r)
J_l(\pi r|\bom|)e^{im\theta}e^{-\ii l\theta}rdrd\theta \right)
\\&=\frac{1}{\sqrt{2\pi N_n^{(m)}(b)}}\sum_{l=-\infty}^{\infty}\ii^{-l}e^{\ii l\Phi(\bom)}
\left(\int_0^{2\pi}e^{\ii m\theta}e^{-\ii l\theta}d\theta\right)\left(\int_0^{b} J_m(b^{-1}z_{nm}r)J_l(\pi r|\bom|)rdr\right)
\\&=\frac{2\pi}{\sqrt{2\pi N_n^{(m)}(b)}}\sum_{l=-\infty}^{\infty}\ii^{-l}e^{\ii l\Phi(\bom)}\delta_{ml}\left(\int_0^{b} J_m(b^{-1}z_{nm}r)J_l(\pi r|\bom|)rdr\right)
\\&=\sqrt{\frac{2\pi}{N_n^{(m)}(b)}}\ii^{-m}e^{\ii m\Phi(\bom)}\left(
\int_0^{b}J_m(b^{-1}z_{nm}r)J_m(\pi r|\bom|)rdr\right).
\end{align*}
Hence, we get
\begin{equation}\label{main.drum.ZPVC}
\widehat{E(\Psi_{nm}^b)}[\bom;b]=\sqrt{\frac{2\pi}{N_n^{(m)}(b)}}\ii^{-m}e^{\ii m\Phi(\bom)}\left(
\int_0^{b}J_m(b^{-1}z_{nm}r)J_m(\pi r|\bom|)rdr\right).
\end{equation}
Using Equation (\ref{2J.a}), we get 
\begin{align*}
\int_0^{b} J_m(b^{-1}z_{nm}r)J_m(\pi r|\bom|)rdr
&=\frac{b^{-1}z_{nm}J_m(\pi|\bom|b)J_m'(z_{nm})-\pi|\bom|J_m(z_{nm})J_m'(\pi|\bom|b)}{b^{-1}(\pi^2|\bom|^2-b^{-2}z_{nm}^2)}
\\&=\frac{z_{nm}J_m(\pi|\bom|b)J_m'(z_{nm})-b\pi|\bom|J_m(z_{nm})J_m'(\pi|\bom|b)}{\pi^2|\bom|^2-b^{-2}z_{nm}^2}
=\frac{z_{nm}J_m(\pi|\bom|b)J_m'(z_{nm})}{\pi^2|\bom|^2-b^{-2}z_{nm}^2}.
\end{align*}
which implies that
\begin{align*}
\widehat{E(\Psi_{nm}^b)}[\bom;b]
=\sqrt{\frac{2\pi}{N_n^{(m)}(b)}}\ii^{-m}e^{\ii m\Phi(\bom)}\frac{z_{nm}J_m(\pi|\bom|b)J_m'(z_{nm})}{\pi^2|\bom|^2-b^{-2}z_{nm}^2}.
\end{align*}
\end{proof}

\begin{corollary}
{\it Let $a>0$ and $b:=a/2$. Suppose $\mathbf{k}:=(k_1,k_2)^T\in\mathbb{R}^2$, $m\in\mathbb{Z}$, and $n\in\mathbb{N}$. We then have
\begin{equation}\label{psi.nm.hat.k.ZPVC}
\widehat{E(\Psi_{nm}^b)}[a^{-1}\mathbf{k};b]=\sqrt{\frac{2\pi}{N_n^{(m)}(b)}}\ii^{-m}e^{\ii m\Phi(\mathbf{k})}\frac{z_{nm}J_m(\pi|\mathbf{k}|/2)J_m'(z_{nm})}{\pi^2a^{-2}|\mathbf{k}|^2-4a^{-2}z_{nm}^2}.
\end{equation}
}\end{corollary}
\begin{proof}
Let $a>0$ and $b:=a/2$. Suppose $\mathbf{k}:=(k_1,k_2)^T\in\mathbb{R}^2$, $m\in\mathbb{Z}$, and $n\in\mathbb{N}$. Using Equation (\ref{psi.nm.hat.omega.ZPVC}), for $\bom:=a^{-1}\mathbf{k}$, we get 
\begin{align*}
\widehat{E(\Psi_{nm}^b)}[a^{-1}\mathbf{k};b]
&=\sqrt{\frac{2\pi}{N_n^{(m)}(b)}}\ii^{-m}e^{\ii m\Phi(\bom)}\frac{z_{nm}J_m(\pi|\bom|b)J_m'(z_{nm})}{\pi^2|\bom|^2-b^{-2}z_{nm}^2}
\\&=\sqrt{\frac{2\pi}{N_n^{(m)}(b)}}\ii^{-m}e^{\ii m\Phi(\mathbf{k})}\frac{z_{nm}J_m(\pi|\mathbf{k}|/2)J_m'(z_{nm})}{\pi^2a^{-2}|\mathbf{k}|^2-4a^{-2}z_{nm}^2}.
\end{align*}
\end{proof}

\begin{proposition}
{\it Let $a>0$, $b:=a/2$, $n,n'\in\mathbb{N}$, and $m,m'\in\mathbb{Z}$.
Then, for each $k\in\mathbb{N}$ and $\ell\in\mathbb{Z}$, we have
\begin{equation}
C_{k,\ell}^a[\Psi_{nm}^b\circledast\Psi_{n'm'}^b]=\frac{2\pi}{\sqrt{N_n^{(m)}(b)N_{n'}^{(m')}(b)}}\ii^{-(m+m')}z_{nm}J_m'(z_{nm})z_{n'm'}J_{m'}'(z_{n'm'})I_{k,\ell}^a(n,n';m,m'),
\end{equation}
with
\begin{equation}
I_{k,\ell}^a(n,n';m,m'):=\sum_{\mathbf{k}\in\mathbb{Z}^2}c_a(\mathbf{k};k,\ell)\frac{e^{\ii(m+m')\Phi(\mathbf{k})}J_m(\pi|\mathbf{k}|/2)J_{m'}(|\mathbf{k}|/2)}{(\pi^2a^{-2}|\mathbf{k}|^2-4a^{-2}z_{nm}^2)(\pi^2a^{-2}|\mathbf{k}|^2-4a^{-2}z_{n'm'}^2)}.
\end{equation}
}\end{proposition}
\begin{proof}
Let $a>0$, $b:=a/2$, $n,n'\in\mathbb{N}$, and $m,m'\in\mathbb{Z}$. Let 
$f:=\Psi_{nm}^b\circledast\Psi_{n'm'}^b$. Then $f$ is a function supported in the disk $\mathbb{B}^2_a$. Suppose
$k\in\mathbb{N}$ and $\ell\in\mathbb{Z}$. Using Equations (\ref{conv.ZBC}) and (\ref{psi.nm.hat.k.ZPVC}), we have
\begin{align*}
C_{k,\ell}^a[\Psi_{nm}^b\circledast\Psi_{n'm'}^b]
&=\sum_{\mathbf{k}\in\mathbb{Z}^2}c_a(\mathbf{k};k,\ell)\widehat{E(\Psi_{nm}^b)}[a^{-1}\mathbf{k};b]\widehat{E(\Psi_{n'm'}^b)}[a^{-1}\mathbf{k};b]
\\&=\frac{2\pi}{\sqrt{N_n^{(m)}(b)N_{n'}^{(m')}(b)}}\ii^{-(m+m')}z_{nm}J_m'(z_{nm})z_{n'm'}J_{m'}'(z_{n'm'})I_{k,\ell}^a(n,n';m,m').
\end{align*}
\end{proof}

\begin{theorem}
Let $a>0$, $b:=a/2$, $n,n'\in\mathbb{N}$, and $m,m'\in\mathbb{Z}$. We then have
\begin{equation}
(\Psi_{nm}^b\circledast\Psi_{n'm'}^b)(r,\theta)=\sum_{\ell=\infty}^\infty\sum_{k=1}^\infty C_{k,\ell}^a[\Psi_{nm}^b\circledast\Psi_{n'm'}^b]\Psi_{k,\ell}^a(r,\theta),\hspace{0.5cm}{\rm for\ all}\ 0\le r\le a,\ 0\le \theta\le 2\pi.
\end{equation}
\end{theorem}

\subsection{Plancherel formula using zero-value boundary condition}

We conclude this section by some Plancherel type formulas involving zero-value boundary condition.

\begin{theorem}
Let $a>0$ and $b:=a/2$. Suppose 
$f:\mathbb{R}^2\to\mathbb{C}$ is a continuous function supported in $\mathbb{B}_{b}^2$. We then have
\begin{equation}\label{PL.ZBVC}
\|f\|^2_{L^2(\mathbb{R}^2)}=\sum_{n=1}^{\infty}\frac{C_{n,0}^{a}[f\ast f^*]}{\sqrt{N_n^{(0)}(a)}},
\end{equation}
where
\begin{equation*}
C_{n,0}^{a}[f\ast f^*]=\frac{\sqrt{\pi}}{2}(-1)^{n}\rho_{n0}\sum_{\mathbf{k}\in\mathbb{Z}^2}\frac{J_{0}(\pi|\mathbf{k}|)}{\pi^2|\mathbf{k}|^2-z_{n0}^2}|\widehat{f}(a^{-1}\mathbf{k})|^2.
\end{equation*}
\end{theorem}
\begin{proof}
Let $a>0$ and $b:=a/2$. Also, let $f:\mathbb{R}^2\to\mathbb{C}$ be a continuous function 
supported in $\mathbb{B}_{b}$. We then have
\begin{equation}\label{2norm.conv}
\|f\|^2_{L^2(\mathbb{R}^2)}=\int_{\mathbb{R}^2}|f(\mathbf{x})|^2d\mathbf{x}=\int_{\mathbb{R}^2}f(\mathbf{x})\overline{f(\mathbf{x})}d\mathbf{x}=f\ast f^*(\mathbf{0}).
\end{equation}
Invoking (\ref{2norm.conv}), we can write
\begin{align*}
\|f\|^2_{L^2(\mathbb{R}^2)}
&=f\ast f^*(\mathbf{0})
\\&=\sum_{m=-\infty}^\infty\sum_{n=1}^{\infty}C_{n,m}^{a}[f\ast f^*]\Psi_{nm}^a(0,0)
\\&=\sum_{m=-\infty}^\infty\sum_{n=1}^{\infty}\frac{C_{n,m}^{a}[f\ast f^*]}{\sqrt{N_n^{(m)}(a)}}J_m(0)
\\&=\sum_{m=-\infty}^\infty\sum_{n=1}^{\infty}\frac{C_{n,m}^{a}[f\ast f^*]}{\sqrt{N_n^{(m)}(a)}}\delta_{m,0}
=\sum_{n=1}^{\infty}\frac{C_{n,0}^{a}[f\ast f^*]}{\sqrt{N_n^{(0)}(a)}}.
\end{align*}
Let $n\in\mathbb{N}$. Using Equation (\ref{ZBC.cts.conv}), we have
\begin{align*}
C_{n,0}^{a}[f\ast f^*]
&=\sum_{\mathbf{k}\in\mathbb{Z}^2}c_a(\mathbf{k};n,0)\widehat{f}(a^{-1}\mathbf{k})\widehat{f^*}(a^{-1}\mathbf{k})
\\&=\sum_{\mathbf{k}\in\mathbb{Z}^2}c_a(\mathbf{k};n,0)\widehat{f}(a^{-1}\mathbf{k})\overline{\widehat{f}(a^{-1}\mathbf{k})}
=2^{-1}\sqrt{\pi}\sum_{\mathbf{k}\in\mathbb{Z}^2}(-1)^{n}\rho_{n0}\frac{J_{0}(\pi|\mathbf{k}|)}{\pi^2|\mathbf{k}|^2-z_{n0}^2}|\widehat{f}(a^{-1}\mathbf{k})|^2.
\end{align*}
\end{proof}

The following formula is the polarized version of Plancherel type formula (\ref{PL.ZBVC}).

\begin{proposition}\label{PL.ZBVC.Polar}
{\it Let $a>0$ and $b:=a/2$. Suppose $f:\mathbb{R}^2\to\mathbb{C}$ be a continuous function supported in $\mathbb{B}_{b}^2$. We then have
\begin{equation}
\|f\|^2_{L^2(\mathbb{R}^2)}=\sum_{n=1}^{\infty}\frac{C_{n,0}^{a}[f\ast f^*]}{\sqrt{N_n^{(0)}(a)}},
\end{equation}
where
\begin{equation}
C_{n,0}^a[f\ast f^*]=\sum_{\tau\in\mathcal{R}}\sum_{\alpha\in\Phi_\tau}A_{n}^a(\tau)|\widehat{f}(a^{-1}\tau\mathbf{u}_\alpha)|^2,
\end{equation}
with
\begin{equation}
A_{n}^a(\tau):=2^{-1}\sqrt{\pi}(-1)^{n}\rho_{n0}\frac{J_{0}(\pi\tau)}{\pi^2\tau^2-z_{n0}^2}.
\end{equation}
}\end{proposition}

\section{\bf Discrete Spectra on Disks using Derivative Boundary Condition}

Throughout this section, for each $m\in\mathbb{Z}$ and $n\in\mathbb{N}$,
we assume that $\rho_{nm}$ are selected with respect to the derivative boundary condition, according to the Sturm-Liouville theory, see Subsection \ref{DBC}.

Next we shall present a unified method for computing the coefficients of
convolution functions, if the basis functions are given by (\ref{2D.B.Fr.g}) with respect to derivative boundary condition.

First, we need some preliminaries results.

\begin{proposition}
{\it Let $a>0$, $0<r,s\le a$, and $0<\alpha,\theta\le 2\pi$. We then have
\begin{equation}\label{JA.ge.DBC}
e^{\ii rs\cos(\alpha-\theta)}=2\sqrt{\pi}\sum_{m=-\infty}^{\infty}\sum_{n=1}^{\infty}
\frac{\ii^m(-1)^nrz_{nm}J_m'(ar)}{(z_{nm}^2-m^2)^{1/2}(\rho_{nm}^2-r^2)}e^{-\ii m\alpha}\Psi_{nm}^a(s,\theta).
\end{equation}
}\end{proposition}
\begin{proof}
Let $a>0$ and $0<r,s\le a$. Also, let $\mathbf{x}:=s\mathbf{u}_\theta$ and $\bom:=r\mathbf{u}_\alpha$. By the Jacobi-Anger expansion, we can write
\begin{equation}\label{JA.e0.DBC}
e^{\ii rs\cos(\alpha-\theta)}=e^{\ii{\bom}\cdot\mathbf{x}}=\sum_{m=-\infty}^{\infty}\ii^m J_m(rs)e^{\ii m\theta}e^{-\ii m\alpha}.
\end{equation}
Let $m\in\mathbb{Z}$ and $0<r\le a$. Expanding $J_m(rs)$ with respect to $s$ as a function over $[0,a]$, using (\ref{mFBS}), we have
\begin{align*}
J_m(rs)=\sum_{n=1}^\infty\left(\int_0^a\mathcal{J}_{nm}^a(p)J_{m}(rp)pdp\right)\mathcal{J}_{nm}^a(s).
\end{align*}
Using (\ref{2J.a}), and since $z_{mn}$ are selected with respect to derivative boundary condition, for each $n\in\mathbb{N}$, we get
\begin{align*}
\int_0^a\mathcal{J}_{nm}^a(p)J_{m}(rp)pdp
&=\frac{1}{\sqrt{N_n^{(m)}(a)}}\int_0^aJ_m(\rho_{nm}p)J_m(rp)pdp
\\&=\frac{a}{\sqrt{N_n^{(m)}(a)}}\frac{rJ_m(\rho_{nm}a)J_m'(ra)-\rho_{nm}J_m(ra )J_m'(\rho_{nm}a)}{(\rho_{nm}^2-r^2)}
\\&=\frac{rJ_m(z_{mn})J_m'(ra)-\rho_{nm}J_m(ra)J_m'(z_{mn})}{\sqrt{D_n^{(m)}}(\rho_{nm}^2-r^2)}
\\&=\frac{rJ_m(z_{mn})J_m'(ra)}{\sqrt{D_n^{(m)}}(\rho_{nm}^2-r^2)}
=\sqrt{2}(-1)^n\frac{}{}\frac{rz_{nm}J_m'(ra)}{(z_{nm}^2-m^2)^{1/2}(\rho_{nm}^2-r^2)}.
\end{align*}
We then deduce that
\begin{equation}\label{JA.ge.alt00.DBC}
J_m(rs)=\sqrt{2}rJ_m'(ar)\sum_{n=1}^\infty\frac{(-1)^nz_{nm}}{(z_{nm}^2-m^2)^{1/2}(\rho_{nm}^2-r^2)}\mathcal{J}_{nm}^a(s).
\end{equation}
Applying Equation (\ref{JA.ge.alt00.DBC}) in (\ref{JA.e0.DBC}), we get
\begin{align*}
e^{\ii rs\cos(\alpha-\theta)}
&=\sum_{m=-\infty}^{\infty}\ii^m J_m(rs)e^{\ii m\theta}e^{-\ii m\alpha}
\\&=\sum_{m=-\infty}^{\infty}\ii^m\left(\sqrt{2}rJ_m'(ar)\sum_{n=1}^\infty\frac{(-1)^nz_{nm}}{(z_{nm}^2-m^2)^{1/2}(\rho_{nm}^2-r^2)}\mathcal{J}_{nm}^a(s)\right)e^{\ii m\theta}e^{-\ii m\alpha}
\\&=\sqrt{2}\sum_{m=-\infty}^{\infty}\sum_{n=1}^\infty\frac{\ii^m(-1)^nrz_{nm}J_m'(ar)}{(z_{nm}^2-m^2)^{1/2}(\rho_{nm}^2-r^2)}\mathcal{J}_{nm}^a(s)e^{\ii m\theta}e^{-\ii m\alpha}
\\&=2\sqrt{\pi}\sum_{m=-\infty}^{\infty}\sum_{n=1}^{\infty}
\frac{\ii^m(-1)^nrz_{nm}J_m'(ar)}{(z_{nm}^2-m^2)^{1/2}(\rho_{nm}^2-r^2)}e^{-\ii m\alpha}\Psi_{nm}^a(s,\theta).
\end{align*}
\end{proof}

\begin{corollary}
{\it Let $a>0$, $\mathbf{k}\in\mathbb{Z}^2$, $n\in\mathbb{N}$, and $m\in\mathbb{Z}$. We then have
\begin{equation}\label{main.alt.DBC}
\int_0^a\int_0^{2\pi}e^{\ii s\mathbf{u}_\theta^T\mathbf{k}}\overline{\Psi_{nm}^a(s\mathbf{u}_\theta)}sdsd\theta=2a\pi\sqrt{\pi}\frac{\ii^m(-1)^n|\mathbf{k}|z_{nm}J_m'(a|\mathbf{k}|)}{(z_{nm}^2-m^2)^{1/2}(z_{nm}^2-\pi^2|\mathbf{k}|^2)}e^{-\ii m\Phi(\mathbf{k})}.
\end{equation}
}\end{corollary}
\begin{proof}
Let $a>0$, $\mathbf{k}\in\mathbb{Z}^2$, $n\in\mathbb{N}$, and $m\in\mathbb{Z}$.
Applying (\ref{JA.ge.DBC}), for $r:=\pi a^{-1}|\mathbf{k}|$ and $\alpha:=\Phi(\mathbf{k})$, we get
\begin{align*}
\int_0^a\int_0^{2\pi}e^{\pi\ii a^{-1}s\mathbf{u}_\theta^T\mathbf{k}}\overline{\Psi_{nm}^a(s\mathbf{u}_\theta)}sdsd\theta
&=\int_0^a\int_0^{2\pi}e^{\pi\ii a^{-1}s|\mathbf{k}|\cos(\Phi(\mathbf{k})-\theta)}\overline{\Psi_{nm}^a(s\mathbf{u}_\theta)}sdsd\theta
\\&=\int_0^a\int_0^{2\pi}e^{\ii rs\cos(\alpha-\theta)}\overline{\Psi_{nm}^a(s\mathbf{u}_\theta)}sdsd\theta
\\&=2\sqrt{\pi}\frac{\ii^m(-1)^nrz_{nm}J_m'(ar)}{(z_{nm}^2-m^2)^{1/2}(\rho_{nm}^2-r^2)}e^{-\ii m\alpha}
\\&=2\sqrt{\pi}\frac{\ii^m(-1)^n\pi a^{-1}|\mathbf{k}|z_{nm}J_m'(\pi|\mathbf{k}|)}{(z_{nm}^2-m^2)^{1/2}(\rho_{nm}^2-\pi^2a^{-2}|\mathbf{k}|^2)}e^{-\ii m\Phi(\mathbf{k})}
\\&=2\sqrt{\pi}\frac{\ii^m(-1)^n\pi a^{-1}|\mathbf{k}|z_{nm}J_m'(\pi|\mathbf{k}|)}{(z_{nm}^2-m^2)^{1/2}(a^{-2}z_{nm}^2-\pi^2a^{-2}|\mathbf{k}|^2)}e^{-\ii m\Phi(\mathbf{k})}
\\&=2\sqrt{\pi}\frac{\ii^m(-1)^n\pi |\mathbf{k}|z_{nm}J_m'(\pi|\mathbf{k}|)e^{-\ii m\Phi(\mathbf{k})}}{a^{-1}(z_{nm}^2-m^2)^{1/2}(z_{nm}^2-\pi^2|\mathbf{k}|^2)}.
\end{align*}
\end{proof}

Next result presents a closed form for coefficients of functions supported in disks, with respect to derivative boundary condition.
\begin{theorem}\label{TH.C.DBC}
Let $a>0$ and $\xi\in L^2(\mathbb{B}_a^2)$ be a function. Also, let $m\in\mathbb{Z}$ and $n\in\mathbb{N}$. We then have
\begin{equation}\label{C.DBC}
C_{n,m}^a(\xi)=\sum_{\mathbf{k}\in\mathbb{Z}^2}c_a'(\mathbf{k};n,m)\widehat{\xi}\{\mathbf{k}\},
\end{equation}
with 
\begin{equation}\label{C.DBC.alt}
\widehat{\xi}\{\mathbf{k}\}:=\int_0^a\int_0^{2\pi}\xi(r,\theta)e^{-\pi\ii a^{-1}r(k_1\cos\theta+k_2\sin\theta)}rdrd\theta,\hspace{1cm}{\rm for\ all}\ \mathbf{k}=(k_1,k_2)^T\in\mathbb{Z}^2,
\end{equation}
and
\begin{equation}
c_a'(\mathbf{k};n,m):=\pi\sqrt{\pi}\frac{\ii^m(-1)^n|\mathbf{k}|\rho_{nm}J_m'(\pi|\mathbf{k}|)e^{-\ii m\Phi(\mathbf{k})}}{2(z_{nm}^2-m^2)^{1/2}(z_{nm}^2-\pi^2|\mathbf{k}|^2)},\hspace{1cm}{\rm for\ all}\ \mathbf{k}=(k_1,k_2)^T\in\mathbb{Z}^2.
\end{equation}
\end{theorem}
\begin{proof}
Let $a>0$ and $\Omega_a:=[-a,a]^2$. Let $\xi\in L^2(\mathbb{B}_a^2)$ be a function and 
$u:=E(\xi)$ be the canonical extension of $\xi$ to the rectangle $\Omega_a$ by zero-padding. Then $u=E(\xi)$ is supported in $\mathbb{B}_a^2$. Also, we have $u=E(\xi)\in L^2(\Omega_a)$ with 
\[
\|E(\xi)\|_{L^2(\Omega_a)}=\|\xi\|_{L^2(\mathbb{B}_a^2)}.
\]
Hence, using the classical Fourier series of the function $u=E(\xi)$, we can write 
\begin{equation}
u(\mathbf{x})
=\frac{1}{4a^2}\sum_{\mathbf{k}\in\mathbb{Z}^2}\widehat{u}(\mathbf{k})e^{\pi\ii a^{-1}\mathbf{x}^T\mathbf{k}},\hspace{1cm}{\rm for\ all}\ \mathbf{x}=(x_1,x_2)^T\in\Omega_a,
\end{equation}
where for the integral vector 
$\mathbf{k}:=(k_1,k_2)\in\mathbb{Z}^2$, we have
\[
\widehat{u}(\mathbf{k})=\int_{-a}^a\int_{-a}^au(x,y)e^{-\pi\ii a^{-1}(k_1x+k_2y)}dxdy.
\]
Since $u=E(\xi)$ is supported in the disk $\mathbb{B}_a^2$, we have 
\begin{align*}
\widehat{u}(\mathbf{k})
=\widehat{\xi}\{\mathbf{k}\}.
\end{align*}
Let $0<s\le a$ and $0\le\theta\le2\pi$. Therefore, for $\mathbf{x}:=s\mathbf{u}_\theta$, we get 
\begin{align*}
\xi(s\mathbf{u}_\theta)
=\frac{1}{4a^2}\sum_{\mathbf{k}\in\mathbb{Z}^2}\widehat{\xi}\{\mathbf{k}\}e^{\pi\ii a^{-1}s\mathbf{u}_\theta^T\mathbf{k}}.
\end{align*}
Hence, using Equation (\ref{main.alt.DBC}), we achieve 
\begin{align*}
C_{n,m}^a(\xi)&=\int_0^a\int_0^{2\pi}\xi(s\mathbf{u}_\theta)\overline{\Psi_{nm}^a(s\mathbf{u}_\theta)}sdsd\theta
\\&=\int_0^a\int_0^{2\pi}\left(\frac{1}{4a^2}\sum_{\mathbf{k}\in\mathbb{Z}^2}\widehat{\xi}\{\mathbf{k}\}e^{\pi\ii a^{-1}s\mathbf{u}_\theta^T\mathbf{k}}\right)\overline{\Psi_{nm}^a(s\mathbf{u}_\theta)}sds d\theta
\\&=\frac{1}{4a^2}\sum_{\mathbf{k}\in\mathbb{Z}^2}\widehat{\xi}\{\mathbf{k}\}\left(\int_0^a\int_0^{2\pi}e^{\pi\ii a^{-1}s\mathbf{u}_\theta^T\mathbf{k}}\overline{\Psi_{nm}^a(s\mathbf{u}_\theta)}sdsd\theta\right)
=\sum_{\mathbf{k}\in\mathbb{Z}^2}c_a'(\mathbf{k};n,m)\widehat{\xi}\{\mathbf{k}\}.
\end{align*}
\end{proof}

\begin{corollary}
{\it Let $a>0$ and $\xi\in L^2(\mathbb{B}_a^2)$ be a function. We then have 
\[
\xi(r,\theta)=\sum_{m=-\infty}^{+\infty}\sum_{n=1}^\infty C_{n,m}^a(\xi)\Psi_{nm}^a(r,\theta),\hspace{0.5cm}{\rm for\ a.e.}\ 0\le r\le a,\ 0\le \theta\le 2\pi.
\]
In particular, if $\xi:\mathbb{B}_a^2\to\mathbb{C}$ is continuous, we have 
\[
\xi(r,\theta)=\sum_{m=-\infty}^{+\infty}\sum_{n=1}^\infty C_{n,m}^a(\xi)\Psi_{nm}^a(r,\theta),\hspace{0.5cm}{\rm for\ all}\ 0\le r\le a,\ 0\le \theta\le 2\pi.
\]
}\end{corollary}

\begin{remark}
The equation (\ref{C.DBC}) guarantees that the coefficients of functions supported in disks with respect to derivative boundary condition, can be computed from the standard Fourier coefficients $\widehat{\xi}\{\mathbf{k}\}$, which can be implemented by FFT.
\end{remark}

Next result gives an explicit closed form for coefficients of zero-padded functions on disks.

\begin{theorem}\label{g.zp.DBC}
{\it Let $a>0$ and $f:\mathbb{R}^2\to\mathbb{C}$ be a square-integrable function 
on the disk $\mathbb{B}_a^2$. We then have
\[
f(r,\theta)=\sum_{m=-\infty}^{+\infty}\sum_{n=1}^\infty C_{n,m}^a[f]\Psi_{nm}^a(r,\theta),\hspace{0.5cm}{\rm for\ a.e.}\ 0\le r\le a,\ 0\le \theta\le 2\pi,
\]
where
\begin{equation}\label{CnmZf.DBC}
C_{n,m}^a[f]:=\sum_{\mathbf{k}\in\mathbb{Z}^2}c_a'(\mathbf{k};n,m)\widehat{f}[a^{-1}\mathbf{k};a],
\hspace{1cm}{\rm for\ all}\ m\in\mathbb{Z},\ n\in\mathbb{N},
\end{equation}
with 
\begin{equation}
\widehat{f}[\bom;a]:=\int_0^a\int_0^{2\pi}f(r,\theta)e^{-\pi\ii (\omega_1\cos\theta+\omega_2\sin\theta)}rdrd\theta,\hspace{1cm}{\rm for\ all}\ \bom=(\omega_1,\omega_2)^T\in\mathbb{R}^2.
\end{equation}
}\end{theorem}
\begin{proof}
Let $a>0$ and $f:\mathbb{R}^2\to\mathbb{C}$ be a function such that 
restriction of $f$ to the disk $\mathbb{B}_a^2$ is square-integrable on 
$\mathbb{B}_a^2$. Let $\xi:=R(f)$
be the restriction of $f$ to the disk $\mathbb{B}_a^2$.
We then have $\xi\in L^2(\mathbb{B}_a^2)$. Hence, we have 
\begin{equation}\label{g.zp.DBC.alt0}
\xi=\sum_{m=-\infty}^{+\infty}\sum_{n=1}^\infty C_{n,m}^a(\xi)\Psi_{nm}^a.
\end{equation}
Suppose $m\in\mathbb{Z}$ and $n\in\mathbb{N}$. Let $C_{n,m}[f]:=C_{n,m}(\xi)$. 
Hence, using Equation (\ref{C.DBC}) for the function $\xi\in L^2(\mathbb{B}_a^2)$, we can write 
\begin{align*}
C_{n,m}^a[f]
=\sum_{\mathbf{k}\in\mathbb{Z}^2}c_a'(\mathbf{k};n,m)\widehat{\xi}\{\mathbf{k}\}.
\end{align*}
Let $\mathbf{k}:=(k_1,k_2)^T\in\mathbb{Z}^2$ be given. Invoking Equation (\ref{C.DBC.alt}), we get 
\begin{align*}
\widehat{\xi}\{\mathbf{k}\}=\widehat{f}[a^{-1}\mathbf{k};a],
\end{align*}
which implies that 
\begin{equation}\label{g.zp.DBC.alt00}
C_{n,m}^a[f]
=\sum_{\mathbf{k}\in\mathbb{Z}^2}c_a'(\mathbf{k};n,m)\widehat{f}[a^{-1}\mathbf{k};a].
\end{equation}
Applying Equation (\ref{g.zp.DBC.alt00}) in (\ref{g.zp.DBC.alt0}) completes the proof.
\end{proof}

\begin{corollary}
{\it Let $a>0$ and $f\in L^2(\mathbb{R}^2)$ be a function. We then have 
\[
f(r,\theta)=\sum_{m=-\infty}^{+\infty}\sum_{n=1}^\infty C_{n,m}^a[f]\Psi_{nm}^a(r,\theta),\hspace{0.5cm}{\rm for\ a.e.}\ 0\le r\le a,\ 0\le \theta\le 2\pi,
\]
}\end{corollary}
\begin{proof}
Let $a>0$ and $f\in L^2(\mathbb{R}^2)$ be a function. Let $\xi:=R(f)$
be the restriction of $f$ to the disk $\mathbb{B}_a^2$.
We then have $\xi\in L^2(\mathbb{B}_a^2)$ with $\|\xi\|_{L^2(\mathbb{B}_a^2)}\le\|f\|_{L^2(\mathbb{R}^2)}$. 
Thus, $f$ is square-integrable on $\mathbb{B}_a^2$. Applying Theorem \ref{g.zp.DBC}, for the function $f$, completes the proof. 
\end{proof}

We then conclude the following results concerning continuous functions. 

\begin{proposition}\label{gr.DBC.pw}
{\it Let $a>0$ and $f:\mathbb{R}^2\to\mathbb{C}$ be a continuous function. 
We then have
\[
f(r,\theta)=\sum_{m=-\infty}^{+\infty}\sum_{n=1}^\infty C_{n,m}^a[f]\Psi_{nm}^a(r,\theta),\hspace{0.5cm}{\rm for\ all}\ 0\le r\le a,\ 0\le \theta\le 2\pi.
\]
}\end{proposition}

\begin{corollary}
{\it Let $a>0$ and $f:\mathbb{R}^2\to\mathbb{C}$ be a continuous function supported in 
$\mathbb{B}_a^2$. We then have 
\[
f(r,\theta)=\sum_{m=-\infty}^{+\infty}\sum_{n=1}^\infty C_{n,m}^a[f]\Psi_{nm}^a(r,\theta),\hspace{0.5cm}{\rm for\ all}\ 0\le r\le a,\ 0\le \theta\le 2\pi,
\]
where
\begin{equation}
C_{n,m}^a[f]:=\sum_{\mathbf{k}\in\mathbb{Z}^2}c_a'(\mathbf{k};n,m)\widehat{f}(a^{-1}\mathbf{k}),
\hspace{1cm}{\rm for\ all}\ m\in\mathbb{Z},\ n\in\mathbb{N}.
\end{equation}
}\end{corollary}
\begin{proof}
Let $a>0$ and $f:\mathbb{R}^2\to\mathbb{C}$ be a continuous function supported in 
$\mathbb{B}_a^2$. Since $f$ is continuous, using Proposition \ref{gr.DBC.pw}, conclude that 
\[
f(r,\theta)=\sum_{m=-\infty}^{+\infty}\sum_{n=1}^\infty C_{n,m}^a[f]\Psi_{nm}^a(r,\theta),\hspace{0.5cm}{\rm for\ all}\ 0\le r\le a,\ 0\le \theta\le 2\pi,
\]
where
\begin{equation}
C_{n,m}^a[f]:=\sum_{\mathbf{k}\in\mathbb{Z}^2}c_a'(\mathbf{k};n,m)\widehat{f}[a^{-1}\mathbf{k};a],
\hspace{1cm}{\rm for\ all}\ m\in\mathbb{Z},\ n\in\mathbb{N}.
\end{equation}
Let $\mathbf{k}:=(k_1,k_2)^T\in\mathbb{Z}^2$ be an integral vector. 
Because $f$ is supported in disk $\mathbb{B}_a^2$, we can write 
\begin{align*}
\widehat{f}[a^{-1}\mathbf{k};a]=\widehat{f}(a^{-1}\mathbf{k}),
\end{align*}
which completes the proof. 
\end{proof}

We then present the following polarized version of Theorem \ref{TH.C.DBC}.

\begin{theorem}
Let $a>0$ with $\xi\in L^2(\mathbb{B}_a^2)$ be a function.
Also, let $m\in\mathbb{Z}$ and $n\in\mathbb{N}$. We then have 
\begin{equation}\label{C.DBC.Polar}
C_{n,m}^a(\xi)=\sum_{\tau\in\mathcal{R}}\sum_{\alpha\in\Phi_\tau}B_{mn}^a(\tau,\alpha)\widehat{\xi}\{\tau\mathbf{u}_\alpha\},
\end{equation}
where 
\begin{equation}
B_{mn}^a(\tau,\alpha):=\pi\sqrt{\pi}\frac{\ii^m(-1)^n\tau \rho_{nm}J_m'(\pi\tau)e^{-\ii m\alpha}}{2(z_{nm}^2-m^2)^{1/2}(z_{nm}^2-\pi^2\tau^2)}.
\end{equation}
\end{theorem}
\begin{proof}
Let $m\in\mathbb{Z}$ and $n\in\mathbb{N}$. First, suppose that $\tau\in\mathcal{R}$ and $\alpha\in\Phi_\tau$. Let $\mathbf{k}:=\tau\mathbf{u}_{\alpha}=(\tau\cos\alpha,\tau\sin\alpha)^T\in\mathbb{Z}^2$. Thus, $|\mathbf{k}|=\tau$ and $\Phi(\mathbf{k})=\alpha$. We then have 
\begin{align*}
c_a'(\tau\mathbf{u}_\alpha;n,m)
&=c_a'(\mathbf{k};n,m)
\\&=\pi\sqrt{\pi}\frac{\ii^m(-1)^n|\mathbf{k}|\rho_{nm}J_m'(\pi|\mathbf{k}|)e^{-\ii m\Phi(\mathbf{k})}}{2(z_{nm}^2-m^2)^{1/2}(z_{nm}^2-\pi^2|\mathbf{k}|^2)}
=\pi\sqrt{\pi}\frac{\ii^m(-1)^n\tau \rho_{nm}J_m'(\pi\tau)e^{-\ii m\alpha}}{2(z_{nm}^2-m^2)^{1/2}(z_{nm}^2-\pi^2\tau^2)}.
\end{align*}
Therefore, using Equation (\ref{C.DBC}), we get 
\begin{align*}
C_{n,m}^a(\xi)&=\sum_{\mathbf{k}\in\mathbb{Z}^2}c_a'(\mathbf{k};n,m)\widehat{\xi}\{\mathbf{k}\}\\&=\sum_{\tau\in\mathcal{R}}\sum_{\alpha\in\Phi_\tau}c_a'(\tau\mathbf{u}_\alpha;n,m)\widehat{\xi}\{\tau\mathbf{u}_\alpha\}
=\sum_{\tau\in\mathcal{R}}\sum_{\alpha\in\Phi_\tau}B_{mn}^a(\tau,\alpha)\widehat{\xi}\{\tau\mathbf{u}_\alpha\}.
\end{align*}
\end{proof}

\begin{corollary}
{\it Let $a>0$ and $\xi:\mathbb{B}_a^2\to\mathbb{C}$ be a continuous function. We then have
\[
\xi(r,\theta)=\sum_{m=-\infty}^{+\infty}\sum_{n=1}^\infty C_{n,m}^a(\xi)\Psi_{nm}^a(r,\theta),\hspace{0.5cm}{\rm for\ all}\ 0\le r\le a,\ 0\le \theta\le 2\pi
\]
where
\begin{equation}
C_{n,m}^a(\xi)=\sum_{\tau\in\mathcal{R}}\sum_{\alpha\in\Phi_\tau}B_{mn}^a(\tau,\alpha)\widehat{\xi}\{\tau\mathbf{u}_\alpha\}.
\end{equation}
}\end{corollary}

Next result gives a polarized version for explicit closed form of coefficients for zero-padded functions.

\begin{proposition}
{\it Let $a>0$ and $f:\mathbb{R}^2\to\mathbb{C}$ be a square-integrable function 
on the disk $\mathbb{B}_a^2$. We then have
\[
f(r,\theta)=\sum_{m=-\infty}^{+\infty}\sum_{n=1}^\infty C_{n,m}^a[f]\Psi_{nm}^a(r,\theta),
\hspace{0.5cm}{\rm for\ a.e.}\ 0\le r\le a,\ 0\le \theta\le 2\pi,\]
with
\begin{equation}
C_{n,m}^a[f]=\sum_{\tau\in\mathcal{R}}\sum_{\alpha\in\Phi_\tau}B_{mn}^a(\tau,\alpha)\widehat{f}[a^{-1}\tau\mathbf{u}_\alpha;a].
\end{equation}
}\end{proposition}

\begin{corollary}
{\it Let $a>0$ and $f:\mathbb{R}^2\to\mathbb{C}$ be a continuous function. 
We then have
\[
f(r,\theta)=\sum_{m=-\infty}^{+\infty}\sum_{n=1}^\infty C_{n,m}^a[f]\Psi_{nm}^a(r,\theta),
\hspace{0.5cm}{\rm for\ all}\ 0\le r\le a,\ 0\le \theta\le 2\pi,
\]
}\end{corollary}

\subsection{Convolution on disks using derivative boundary condition}
We then continue by investigating analytical aspects of approximations for convolution of functions supported in disks, with respect to derivative boundary condition.

The following theorem introduces a constructive method for computing the coefficients of zero-padded convolutions of functions on disks, with respect to derivative boundary condition.

\begin{theorem}
Let $a>0$ and $b:=a/2$. Suppose $f_j:\mathbb{R}^2\to\mathbb{C}$ with $j\in\{1,2\}$ are  
square integrable functions on $\mathbb{B}_b^2$. Let $m\in\mathbb{Z}$ and $n\in\mathbb{N}$. We then have
\begin{equation}\label{conv.DBC}
C_{n,m}^a[f_1\circledast f_2]=\sum_{\mathbf{k}\in\mathbb{Z}^2}c_a'(\mathbf{k};n,m)\widehat{f_1}[a^{-1}\mathbf{k};b]\widehat{f_2}[a^{-1}\mathbf{k};b],
\end{equation}
where
\begin{equation}\label{zp.hat.DBC}
\widehat{f_j}[\bom;b]:=\int_0^b\int_0^{2\pi}f_j(r,\theta)e^{-\pi\ii r(\omega_1\cos\theta+\omega_2\sin\theta)}rdrd\theta,\hspace{1cm}{\rm for\ all}\ \bom:=(\omega_1,\omega_2)^T\in\mathbb{R}^2.
\end{equation}
\end{theorem}
\begin{proof}
Let $f_j:\mathbb{R}^2\to\mathbb{C}$ with $j\in\{1,2\}$ be square integrable functions on $\mathbb{B}_b^2$. Let $\xi_j:=R(f_j)$ be the restriction of $f_j$ to the disk 
$\mathbb{B}_b^2$ and $E(\xi_j)$ be the extension of $\xi_j$ to $\mathbb{R}^2$ by zero-padding.  Thus, we get $\xi_j\in L^2(\mathbb{B}_b^2)$ and hence $E(\xi_j)\in L^2(\mathbb{R}^2)$. Since $L^2(\mathbb{B}_b^2)\subseteq L^1(\mathbb{B}_b^2)$, we get $\xi_j\in L^1(\mathbb{B}_b^2)$ as well. Therefore, we have $E(\xi_j)\in L^1\cap L^2(\mathbb{R}^2)$. By definition of zero-padded convolutions, we have  
\[
f_1\circledast f_2=E(\xi_1)\ast E(\xi_2).
\]
Using Prosposition \ref{222.R2}, we deduce that $f:=f_1\circledast f_2$ is square integrable on $\mathbb{B}_a^2$. Let $\mathbf{k}\in\mathbb{Z}^2$ be given. Since 
$f=f_1\circledast f_2$ is supported in the disk 
$\mathbb{B}_a^2$, we have  
\[
\widehat{f}[a^{-1}\mathbf{k};a]=\widehat{f}(a^{-1}\mathbf{k}).
\]
By the convolution property of Fourier transform, we get 
\begin{equation}\label{conv.zp.0'}
\widehat{f}[a^{-1}\mathbf{k};a]=\widehat{E(\xi_1)}(a^{-1}\mathbf{k})\widehat{E(\xi_2)}(a^{-1}\mathbf{k}).
\end{equation}
Since each $E(\xi_j)$ is supported in the disk $\mathbb{B}_b^2$, we get 
\begin{equation}\label{conv.zp.00'}
\widehat{E(\xi_j)}(a^{-1}\mathbf{k})=\widehat{f_j}[a^{-1}\mathbf{k};b].
\end{equation}
Applying Equation (\ref{conv.zp.00'}) in Equation (\ref{conv.zp.0'}), we get 
\begin{equation}\label{conv.zp.000'}
\widehat{f}[a^{-1}\mathbf{k};a]=\widehat{f_1}[a^{-1}\mathbf{k};b]\widehat{f_2}[a^{-1}\mathbf{k};b].
\end{equation}
Let $m\in\mathbb{Z}$ and $n\in\mathbb{N}$.
Then, using Equation (\ref{conv.zp.000'}) in Equation (\ref{CnmZf.DBC}), we get
\begin{align*}
C_{n,m}^a[f_1\circledast f_2]
&=C_{n,m}^a[f]
\\&=\sum_{\mathbf{k}\in\mathbb{Z}^2}c_a'(\mathbf{k};n,m)\widehat{f}[a^{-1}\mathbf{k};a]
\\&=\sum_{\mathbf{k}\in\mathbb{Z}^2}c_a'(\mathbf{k};n,m)\widehat{f_1}[a^{-1}\mathbf{k};b]\widehat{f_2}[a^{-1}\mathbf{k};b],
\end{align*}
which completes the proof.
\end{proof}

\begin{corollary}
{\it Let $a>0$ and $b:=a/2$. Suppose $f_j:\mathbb{R}^2\to\mathbb{C}$ with $j\in\{1,2\}$ are functions integrable on $\mathbb{B}_b^2$. We then have
\begin{equation}
(f_1\circledast f_2)(r,\theta)=\sum_{m=-\infty}^\infty\sum_{n=1}^\infty C_{n,m}^{a}[f_1\circledast f_2]\Psi_{nm}^a(r,\theta),\hspace{0.5cm}{\rm for\ a.e.}\ 0\le r\le a,\ 0\le \theta\le 2\pi,
\end{equation}
where
\begin{equation}
C_{n,m}^a[f_1\circledast f_2]=\sum_{\mathbf{k}\in\mathbb{Z}^2}c_a'(\mathbf{k};n,m)\widehat{f_1}[a^{-1}\mathbf{k};b]\widehat{f_2}[a^{-1}\mathbf{k};b].
\end{equation}
}\end{corollary}

\begin{proposition}
{\it Let $a>0$ and $b:=a/2$. Suppose $f_j:\mathbb{R}^2\to\mathbb{C}$ with $j\in\{1,2\}$ 
are continuous functions supported in $\mathbb{B}^2_{b}$. We then have
\begin{equation}\label{2D.EX.Fr.ast.DBC.cts}
(f_1\ast f_2)(r,\theta)=\sum_{m=-\infty}^\infty\sum_{n=1}^{\infty} C_{n,m}^{a}[f_1\ast f_2]\Psi_{nm}^a(r,\theta),\hspace{0.5cm}{\rm for\ all}\ 0\le r\le a,\ 0\le \theta\le 2\pi
\end{equation}
where
\begin{equation}
C_{n,m}^a[f_1\ast f_2]=\sum_{\mathbf{k}\in\mathbb{Z}^2}c_a'(\mathbf{k};n,m)\widehat{f_1}(a^{-1}\mathbf{k})\widehat{f_2}(a^{-1}\mathbf{k}).
\end{equation}
}\end{proposition}

We then present the following polarized version of closed forms for coefficients of zero-padded convolutions on disks with respect to derivative boundary condition.

\begin{theorem}
Let $a>0$ and $b:=a/2$. Suppose $f_j:\mathbb{R}^2\to\mathbb{C}$ with $j\in\{1,2\}$ are  continuous functions. Also, let $m\in\mathbb{Z}$ and $n\in\mathbb{N}$. We then have
\begin{equation}
C_{n,m}^a[f_1\circledast f_2]=\sum_{\tau\in\mathcal{R}}\sum_{\alpha\in\Phi_\tau}B_{mn}^a(\tau,\alpha)\widehat{f_1}[a^{-1}\tau\mathbf{u}_\alpha;b]\widehat{f_2}[a^{-1}\tau\mathbf{u}_\alpha;b].
\end{equation}
In particular, if $f_j:\mathbb{R}^2\to\mathbb{C}$ with $j\in\{1,2\}$ are continuous functions supported in $\mathbb{B}^2_{b}$, we have
\begin{equation}
C_{n,m}^a[f_1\circledast f_2]=\sum_{\tau\in\mathcal{R}}\sum_{\alpha\in\Phi_\tau}B_{mn}^a(\tau,\alpha)\widehat{f_1}(a^{-1}\tau\mathbf{u}_\alpha)\widehat{f_2}(a^{-1}\tau\mathbf{u}_\alpha).
\end{equation}
\end{theorem}

\subsection{Convolution of basis elements on disks using derivative boundary condition}

Let $a>0$ and $b:=a/2$. Let $f_j:\mathbb{R}^2\to\mathbb{C}$ with $j\in\{1,2\}$ be 
functions square integrable on the disk $\mathbb{B}_{b}^2$ with the associated derivative boundary condition coefficients
$\left\{C_{n,m}^{b}[f_j]:n\in\mathbb{N},m\in\mathbb{Z}\right\}$. Hence, we can write
\begin{equation}
\xi_j=\sum_{m=-\infty}^\infty\sum_{n=1}^\infty C_{n,m}^{b}[f_j]\Psi_{nm}^{b},
\end{equation}
where $\xi_j:=R(f_j)$ is the restriction of $f_j$ into the disk $\mathbb{B}_b^2$ and 
\[
C_{n,m}^{b}[f_j]=\sum_{\mathbf{k}\in\mathbb{Z}^2}c_a'(\mathbf{k};n,m)\widehat{f_j}[a^{-1}\mathbf{k};b],
\]
for $m\in\mathbb{Z}$ and $n\in\mathbb{N}$.

Thus, we get 
\[
E(\xi_j)=\sum_{m=-\infty}^\infty\sum_{n=1}^\infty C_{n,m}^{b}[f_j]E(\Psi_{nm}^{b}).
\]
Using linearity of convolutions, as linear operators, we get
\begin{align*}
f_1\circledast f_2 
&=E(\xi_1)\ast E(\xi_2)
\\&=\left(\sum_{m=-\infty}^\infty\sum_{n=1}^\infty C_{n,m}^{b}[f_1]E(\Psi_{nm}^{b})\right)\ast\left(\sum_{m'=-\infty}^\infty\sum_{n'=1}^\infty C_{n',m'}^{b}[f_2]E(\Psi_{n'm'}^{b})\right)
\\&=\sum_{m=-\infty}^\infty\sum_{n=1}^\infty\sum_{m'=-\infty}^\infty\sum_{n'=1}^\infty C_{n,m}^{b}[f_1]C_{n',m'}^{b}[f_2]E(\Psi_{nm}^{b})\ast E(\Psi_{n'm'}^{b})
\\&=\sum_{m=-\infty}^\infty\sum_{n=1}^\infty\sum_{m'=-\infty}^\infty\sum_{n'=1}^\infty C_{n,m}^{b}[f_1]C_{n',m'}^{b}[f_2]\Psi_{nm}^b\circledast\Psi_{n'm'}^b.
\end{align*}

Hence, convolution of circular drums (basis elements) can be viewed as pre-computed kernels.

\begin{proposition}
{\it Let $a>0$ and $b:=a/2$. Suppose $\bom:=(\omega_1,\omega_2)^T\in\mathbb{R}^2$, $m\in\mathbb{Z}$, and $n\in\mathbb{N}$. We then have
\begin{equation}\label{psi.nm.hat.omega.DBC}
\widehat{E(\Psi_{nm}^b)}[\bom;b]=\sqrt{\frac{2\pi}{N_n^{(m)}(a/2)}}\ii^{-m}e^{\ii m\Phi(\mathbf{k})}\frac{z_{nm}J_m(\pi|\mathbf{k}|/2)J_m'(z_{nm})}{\pi^2a^{-2}|\mathbf{k}|^2-4a^{-2}z_{nm}^2}.
\end{equation}
}\end{proposition}
\begin{proof}
Let $m\in\mathbb{Z}$ and $n\in\mathbb{N}$. Applying Equation (\ref{zp.hat.DBC}), we get
\begin{equation}\label{main.drum.DBC}
\widehat{E(\Psi_{nm}^b)}[\bom;b]=\sqrt{\frac{2\pi}{N_n^{(m)}(b)}}\ii^{-m}e^{\ii m\Phi(\bom)}\left(
\int_0^{b}J_m(b^{-1}z_{nm}r)J_m(\pi r|\bom|)rdr\right).
\end{equation}
Using (\ref{2J.a}), we have  
\begin{align*}
\int_0^{b} J_m(b^{-1}z_{nm}r)J_m(\pi r|\bom|)rdr
&=\frac{b^{-1}z_{nm}J_m(\pi|\bom|b)J_m'(z_{nm})-\pi|\bom|J_m(z_{nm})J_m'(\pi|\bom|b)}{b^{-1}(\pi^2|\bom|^2-b^{-2}z_{nm}^2)}
\\&=\frac{z_{nm}J_m(\pi|\bom|b)J_m'(z_{nm})-b\pi|\bom|J_m(z_{nm})J_m'(\pi|\bom|b)}{\pi^2|\bom|^2-b^{-2}z_{nm}^2}
=\frac{-b\pi|\bom|J_m(z_{nm})J_m'(\pi|\bom|b)}{\pi^2|\bom|^2-b^{-2}z_{nm}^2}.
\end{align*}
which implies that
\begin{align*}
\widehat{E(\Psi_{nm}^b)}[\bom;b]
=\sqrt{\frac{2\pi}{N_n^{(m)}(b)}}\ii^{-m}e^{\ii m\Phi(\bom)}\frac{-b\pi|\bom|J_m(z_{nm})J_m'(\pi|\bom|b)}{\pi^2|\bom|^2-b^{-2}z_{nm}^2}.
\end{align*}
\end{proof}

\begin{proposition}
{\it Let $a>0$, $\mathbf{k}\in\mathbb{Z}^2$, $m\in\mathbb{Z}$, and $n\in\mathbb{N}$. We then have
\begin{equation}
\widehat{E(\Psi_{nm}^b)}[a^{-1}\mathbf{k};b]=\sqrt{\frac{2\pi}{N_n^{(m)}(b)}}\ii^{-m}e^{\ii m\Phi(\mathbf{k})}\frac{-\pi|\mathbf{k}|J_m(z_{nm})J_m'(\pi|\mathbf{k}|/2)}{2(\pi^2a^{-2}|\mathbf{k}|^2-4a^{-2}z_{nm}^2)}.
\end{equation}
}\end{proposition}
\begin{proof}
Let $a>0$ and $b:=a/2$. Suppose $\mathbf{k}:=(k_1,k_2)^T\in\mathbb{R}^2$, $m\in\mathbb{Z}$, and $n\in\mathbb{N}$. Using Equation (\ref{psi.nm.hat.omega.DBC}), for $\bom:=a^{-1}\mathbf{k}$, we get 
\begin{align*}
\widehat{E(\Psi_{nm}^b)}[a^{-1}\mathbf{k};b]
&=\sqrt{\frac{2\pi}{N_n^{(m)}(b)}}\ii^{-m}e^{\ii m\Phi(\bom)}\frac{-b\pi|\bom|J_m(z_{nm})J_m'(\pi|\bom|b)}{\pi^2|\bom|^2-b^{-2}z_{nm}^2}
\\&=\sqrt{\frac{2\pi}{N_n^{(m)}(b)}}\ii^{-m}e^{\ii m\Phi(\mathbf{k})}\frac{-b\pi a^{-1}|\mathbf{k}|J_m(z_{nm})J_m'(\pi a^{-1}|\mathbf{k}|b)}{\pi^2a^{-2}|\mathbf{k}|^2-b^{-2}z_{nm}^2}
\\&=\sqrt{\frac{2\pi}{N_n^{(m)}(b)}}\ii^{-m}e^{\ii m\Phi(\mathbf{k})}\frac{-\pi|\mathbf{k}|J_m(z_{nm})J_m'(\pi|\mathbf{k}|/2)}{2(\pi^2a^{-2}|\mathbf{k}|^2-4a^{-2}z_{nm}^2)}.
\end{align*}
\end{proof}

\begin{proposition}
{\it Let $a>0$, $b:=a/2$, $n,n'\in\mathbb{N}$, and $m,m'\in\mathbb{Z}$.
Then, for each $k\in\mathbb{N}$ and $\ell\in\mathbb{Z}$, we have
\begin{equation}
C_{k,\ell}^a[\Psi_{nm}^b\circledast\Psi_{n'm'}^b]=\frac{2\pi}{\sqrt{N_n^{(m)}(b)N_{n'}^{(m')}(b)}}\ii^{-(m+m')}z_{nm}J_m'(z_{nm})z_{n'm'}J_{m'}'(z_{n'm'})I_{k,\ell}^a(n,n';m,m'),
\end{equation}
with
\begin{equation}
H_{k,\ell}^a(n,n';m,m'):=\sum_{\mathbf{k}\in\mathbb{Z}^2}c_a'(\mathbf{k};k,\ell)\frac{e^{\ii(m+m')\Phi(\mathbf{k})}
\pi^2|\mathbf{k}|^2J_m(z_{nm})J_m'(\pi|\mathbf{k}|/2)J_{m'}(z_{n'm'})J_{m'}'(\pi|\mathbf{k}|/2)}{4(\pi^2a^{-2}|\mathbf{k}|^2-4a^{-2}z_{nm}^2)(\pi^2a^{-2}|\mathbf{k}|^2-4a^{-2}z_{n'm'}^2)}.
\end{equation}
}\end{proposition}
\begin{proof}
Let $a>0$, $n,n'\in\mathbb{N}$, and $m,m'\in\mathbb{Z}$. Let $f:=\Psi_{nm}^b\circledast\Psi_{n'm'}^b$. Then, $f$ is a function supported in the disk $\mathbb{B}^2_a$. Suppose
$k\in\mathbb{N}$ and $\ell\in\mathbb{Z}$. Using Equation (\ref{conv.DBC}), we have
\begin{align*}
C_{k,\ell}^a[\Psi_{nm}^b\circledast\Psi_{n'm'}^b]
&=\sum_{\mathbf{k}\in\mathbb{Z}^2}c_a'(\mathbf{k};k,\ell)\widehat{E(\Psi_{nm}^b)}[a^{-1}\mathbf{k};b]\widehat{E(\Psi_{n'm'}^b)}[a^{-1}\mathbf{k};b]
\\&=\frac{2\pi}{\sqrt{N_n^{(m)}(b)N_{n'}^{(m')}(b)}}\ii^{-(m+m')}z_{nm}J_m'(z_{nm})z_{n'm'}J_{m'}'(z_{n'm'})H_{k,\ell}^a(n,n';m,m').
\end{align*}
\end{proof}

\begin{theorem}
Let $a>0$, $n,n'\in\mathbb{N}$, and $m,m'\in\mathbb{Z}$. We then have
\begin{equation}
(\Psi_{nm}^b\circledast\Psi_{n'm'}^b)(r,\theta)=\sum_{\ell=\infty}^\infty\sum_{k=1}^\infty C_{k,\ell}^a[\Psi_{nm}^b\circledast\Psi_{n'm'}^b]\Psi_{k,\ell}^a(r,\theta),
\end{equation}
where for each $\ell\in\mathbb{Z}$ and $k\in\mathbb{N}$, we have
\begin{equation}
C_{k,\ell}^a[\Psi_{nm}^b\circledast\Psi_{n'm'}^b]=\frac{2\pi}{\sqrt{N_n^{(m)}(b)N_{n'}^{(m')}(b)}}\ii^{-(m+m')}z_{nm}J_m'(z_{nm})z_{n'm'}J_{m'}'(z_{n'm'})H_{k,\ell}^a(n,n';m,m').
\end{equation}
\end{theorem}

\subsection{Plancherel formula using derivative boundary condition}

We conclude this section by some Plancherel type formulas involving derivative boundary condition.

\begin{theorem}
Let $a>0$ and $b:=a/2$. Suppose 
$f:\mathbb{R}^2\to\mathbb{C}$ is a continuous function supported in $\mathbb{B}_{b}^2$. We then have
\begin{equation}\label{PL.DBC}
\|f\|^2_{L^2(\mathbb{R}^2)}=\sum_{n=1}^{\infty}\frac{C_{n,0}^{a}[f\ast f^*]}{\sqrt{N_n^{(0)}(a)}},
\end{equation}
where
\begin{equation}
C_{n,0}^{a}[f\ast f^*]:=2^{-1}\pi\sqrt{\pi}\frac{(-1)^n}{a}\sum_{\mathbf{k}\in\mathbb{Z}^2}\frac{|\mathbf{k}|J_0'(\pi|\mathbf{k}|)}{(z_{n0}^2-\pi^2|\mathbf{k}|^2)}|\widehat{f}(a^{-1}\mathbf{k})|^2.
\end{equation}
\end{theorem}
\begin{proof}
Let $a>0$ and $f\in L^2(\mathbb{R}^2)$ with compact support in $\mathbb{B}_{a/2}^2$.
Invoking (\ref{2norm.conv}), we can write
\begin{align*}
\|f\|^2_{L^2(\mathbb{R}^2)}
=\sum_{n=1}^{\infty}\frac{C_{n,0}^{a}[f\ast f^*]}{\sqrt{N_n^{(0)}(a)}}.
\end{align*}
Let $n\in\mathbb{N}$ and $\mathbf{k}\in\mathbb{Z}^2$ be given. Then, we have
\[
c_a'(\mathbf{k};n,0)=2^{-1}\pi\sqrt{\pi}\frac{(-1)^n|\mathbf{k}|z_{n0}J_0'(\pi|\mathbf{k}|)}{a(z_{n0}^2-\pi^2\mathbf{k}^2)}=2^{-1}\pi\sqrt{\pi}\frac{(-1)^n|\mathbf{k}|J_0'(\pi|\mathbf{k}|)}{a(z_{n0}^2-\pi^2|\mathbf{k}|^2)}.
\]
Using (\ref{conv.DBC}), we have
\begin{align*}
C_{n,0}^{a}[f\ast f^*]
&=\sum_{\mathbf{k}\in\mathbb{Z}^2}c_a'(\mathbf{k};n,0)\widehat{f}(a^{-1}\mathbf{k})\widehat{f^*}(a^{-1}\mathbf{k})
=\sum_{\mathbf{k}\in\mathbb{Z}^2}c_a'(\mathbf{k};n,0)\widehat{f}(a^{-1}\mathbf{k})\overline{\widehat{f}(a^{-1}\mathbf{k})}
\\&=\sum_{\mathbf{k}\in\mathbb{Z}^2}c_a'(\mathbf{k};n,0)|\widehat{f}(a^{-1}\mathbf{k})|^2
=2^{-1}\pi\sqrt{\pi}\frac{(-1)^n}{a}\sum_{\mathbf{k}\in\mathbb{Z}^2}\frac{|\mathbf{k}|J_0'(\pi|\mathbf{k}|)}{(z_{n0}^2-\pi^2|\mathbf{k}|^2)}|\widehat{f}(a^{-1}\mathbf{k})|^2,
\end{align*}
which completes the proof.
\end{proof}

The following formula is the polarized version of Plancherel type formula (\ref{PL.DBC}).

\begin{proposition}\label{PL.DBC.Polar}
{\it Let $a>0$ and $b:=a/2$. Suppose 
$f:\mathbb{R}^2\to\mathbb{C}$ is a continuous function supported in $\mathbb{B}_{b}^2$. We then have
\begin{equation}
\|f\|^2_{L^2(\mathbb{R}^2)}=\sum_{n=1}^{\infty}\frac{C_{n,0}^{a}[f\ast f^*]}{\sqrt{N_n^{(0)}(a)}},
\end{equation}
with
\begin{equation}
C_{n,0}^a[f\ast f^*]=\sum_{\tau\in\mathcal{R}}\sum_{\alpha\in\Phi_\tau}B_{n}^a(\tau)|\widehat{f}(a^{-1}\tau\mathbf{u}_\alpha)|^2,
\hspace{1cm}B_{n}^a(\tau):=2^{-1}\pi\sqrt{\pi}\frac{(-1)^n\tau J_0'(\pi\tau)}{a(z_{n0}^2-\pi^2\tau^2)}.
\end{equation}
}\end{proposition}

{\bf Concluding remarks.}
In many applications, convolutions of compactly supported functions on the Euclidean plane are required. But in classical Fourier theory,
the spectrum of such functions is neither discrete nor compactly supported. Here an alternative theory is put forth in which Fourier-Bessel expansions are used for functions supported on disks. The closure properties of such expansions under convolution are derived. In contrast
to the usual approach of periodizing to discretize the spectrum, followed by using the FFT, the approach presented here has several  potential benefits. In particular, there is a natural way to rotate such expansions in a way that is mathematically exact. Consequentially, rotation-invariant descriptors of images can be computed naturally. In contrast, periodization and the FFT destroys any rotation invariance. 
It is shown that applying different boundary conditions imply to different spectra of the system. Therefore, the selection of the boundary condition depends on the nature of the problem.

\

{\bf Acknowledgments.}
This work has been supported by the National Institute of General Medical Sciences of the NIH under award number R01GM113240, by the US National Science Foundation under grant NSF CCF-1640970, and by Office of Naval Research Award N00014-17-1-2142. The authors gratefully acknowledge the supporting agencies. The findings and opinions expressed here are only those of the authors, and not of the funding agencies.
\vspace{0.1cm}
\bibliographystyle{amsplain}

\begin{thebibliography}{10}
\bibitem{SLT.IA} M.A. Al-Gwaiz, \textit{Sturm-Liouville Theory and Its Applications}, Springer-Verlag London, Ltd., London, 2008. x+264 pp.

\bibitem{PIb2} G.S. Chirikjian, \textit{Stochastic Models, Information Theory, and Lie Groups}, Vol. 2. Analytic Methods and Modern Applications,
Applied and Numerical Harmonic Analysis. Birkh\"auser/Springer, New York, xxviii+433 pp, 2012.

\bibitem{PIb4} G.S. Chirikjian, \textit{Stochastic Models, Information Theory, and Lie Groups}, Vol. 1. Classical Results and Geometric Methods, Applied and Numerical Harmonic Analysis. Birkh\"auser Boston, Inc., Boston, MA, xxii+380 pp, 2009.

\bibitem{PI5} G.S. Chirikjian, A.B. Kyatkin, \textit{Harmonic Analysis for Engineers and Applied Scientists: Updated and Expanded Edition}, Courier Dover Publications, 2016.

\bibitem{PI6} G.S. Chirikjian, A.B. Kyatkin, \textit{Engineering Applications of Noncommutative Harmonic Analysis}, CRC Press, 2000.

\bibitem{Fei0} H. G. Feichtinger, \textit{On a new Segal algebra,}
Monatsh. Math. 92 (1981), no. 4, 269-289.

\bibitem{Fei1} H. G. Feichtinger, \textit{Banach convolution algebras of functions II,}
  Monatsh. Math. 87 (1979), no. 3, 181-207.

\bibitem{Fei2} H. G. Feichtinger, \textit{On a class of convolution algebras of functions,}
 Ann. Inst. Fourier (Grenoble) 27 (1977), no. 3, vi, 135-162.

\bibitem{Fei.Gro1} H. G. Feichtinger, K.H. Gr\"ochenig,
\textit{Banach spaces related to integrable group representations and their atomic decompositions. I.},
J. Funct. Anal. 86 (1989), no. 2, 307--340.
%
\bibitem{Fei.Gro2} H. G. Feichtinger, K.H. Gr\"ochenig,
\textit{Banach spaces related to integrable group representations and their atomic decompositions. II.}, Monatsh. Math. 108 (1989), no. 2-3, 129--148.

\bibitem{AGHF.BBMS-NS} A. Ghaani Farashahi, \textit{Abstract Banach convolution function modules over coset spaces of compact subgroups in locally compact groups}, Bulletin of the Brazilian Mathematical Society, New Series (2019), https://doi.org/10.1007/s00574-018-00129-6.

\bibitem{AGHF.GSG}  A. Ghaani Farashahi, G.S. Chirikjian, \textit{Fourier-Zernike series of convolutions on disks}, Mathematics, 2018, 6(12):290.

\bibitem{AGHF.CJM} A. Ghaani Farashahi, \textit{A class of abstract linear representations for convolution function algebras over homogeneous spaces of compact groups}, Canad. J. Math. 70, no. 1, 97-116, 2018.

\bibitem{AGHF.JAuMS} A. Ghaani Farashahi, \textit{Abstract harmonic analysis of relative convolutions over canonical homogeneous spaces of semidirect product groups}, J. Aust. Math. Soc. 101, no. 2, 171-187, 2016.

\bibitem{AGHF.IJM.2015} A. Ghaani Farashahi, \textit{Abstract convolution function algebras over homogeneous spaces of compact groups}, Illinois J. Math. 59, no. 4, 1025-1042, 2015.

\bibitem{Kavraki} L.E. Kavraki, \textit{Computation of configuration-space obstacles using the fast Fourier transform}, IEEE Transactions on Robotics and Automation, 11(3), 1995.

\bibitem{Kim.Kim} W.Y. Kim and Y. S. Kim, \textit{Robust Rotation Angle Estimator}, IEEE Transaction on Pattern Analysis and Machine Intelligence, 21(8):768-773, 1999.

\bibitem{Kisil.Adv} V. Kisil, \textit{Relative convolutions. I. Properties and applications}, Adv. Math. 147, no. 1, 35-73, 1999.

\bibitem{Kisil.BJMA} V. Kisil, \textit{Calculus of operators: covariant transform and relative convolutions}, Banach J. Math. Anal. 8, no. 2, 156-184, 2014.

\bibitem{Kya.PI.2000} A.B. Kyatkin, G.S. Chirikjian, \textit{Algorithms for fast convolutions on motion groups}, Appl.Comput.Harmon.Anal.9:220-241,2000.

\bibitem{Kya.PI.1999} A.B. Kyatkin, G.S. Chirikjian, \textit{Computation of robot configuration and workspaces via the Fourier transform on the discrete motion group}, International Journal of Robotics Research. 18(6):601-615. 1999.

\bibitem{50} H. Reiter, J. D. Stegeman, \textit{Classical Harmonic Analysis and Locally Compact Groups}, 2nd Ed, Oxford University Press, New York, 2000.

\bibitem{19Venkatraman2009}
V. Venkatraman, L. Sael, D. Kihara, \textit{Potential for Protein Surface Shape Analysis Using Spherical Harmonics and 3D Zernike Descriptors}, Cell Biochemistry and Biophysics, 54(1-3):23--32, 2009.

\bibitem{19Venkatraman2009a}
V. Venkatraman, Y.D. Yang, L. Sael, D. Kihara, \textit{Protein--Protein Docking Using Region--Based 3D Zernike Descriptors}, BMC Bioinformatics, 10(1):407, 2009.

\bibitem{Zet.An} A. Zettl, \textit{Sturm–Liouville Theory}, Mathematical Surveys and Monographs, 121. American Mathematical Society, Providence, RI, 2005. xii+328 pp.


\end{thebibliography}

\end{document}